\documentclass[12pt,draftclsnofoot,onecolumn]{IEEEtran}

\usepackage{amsmath}
\usepackage{mathrsfs}
\usepackage{amssymb}
\usepackage{amsthm}

\usepackage{graphicx}
\let\labelindent\relax
\usepackage{enumitem}
\usepackage{epstopdf}
\usepackage{subfigure}
\usepackage{color}
\usepackage{multirow}
 
\usepackage{verbatim}
\usepackage{algorithm,algorithmicx,algpseudocode}
\usepackage{cite}
\usepackage{bbm}
\usepackage{bigstrut}
\usepackage[linktocpage,colorlinks,linkcolor=blue,anchorcolor=blue,citecolor=blue,urlcolor=blue]{hyperref}
\usepackage{cleveref}

\usepackage{tikz}

\makeatletter
\newcommand\footnoteref[1]{\protected@xdef\@thefnmark{\ref{#1}}\@footnotemark}
\makeatother

\usepackage{textcase}
\usepackage[tablename=TABLE]{caption}
\DeclareCaptionTextFormat{up}{\MakeTextUppercase{#1}}
\numberwithin{equation}{section}
\crefname{equation}{equation}{equations}
\crefname{section}{Section}{Sections}

\newtheorem{theorem}{Theorem}%[section]
\crefname{theorem}{Theorem}{Theorems}
\crefname{definition}{Definition}{Definitions}

\newtheorem{lemma}{Lemma}%[theorem]
\crefname{lemma}{Lemma}{Lemmas}

\newtheorem{proposition}{Proposition}%[theorem]
\crefname{proposition}{Proposition}{Propositions}

\newtheorem{definition}{Definition}%[theorem]
\crefname{corollary}{Corollary}{Corollaries}

\newtheorem{assumption}{Assumption}%[section]
\crefname{assumption}{Assumption}{Assumptions}
\crefname{algorithm}{Algorithm}{Algorithms}

\newtheorem{remark}{Remark}
\crefname{remark}{Remark}{Remark}

\newcommand{\G}{\mathcal{G}}

\newcommand{\cL}{\mathcal{L}}
\newcommand{\M}{\mathcal{M}}
\newcommand{\N}{\mathcal{N}}

\newcommand{\cS}{\mathcal{S}}
\newcommand{\cX}{\mathcal{X}}

\newcommand{\R}{\mathbb{R}}

\newcommand{\etal}{ et al. }

\newcommand{\half}{\frac{1}{2}}

\newcommand{\X}{\mathbf{x}}
\newcommand{\Y}{\mathbf{y}}

\newcommand{\be}{\begin{equation}}
\newcommand{\ee}{\end{equation}}
\newcommand{\ba}{\begin{array}}
\newcommand{\ea}{\end{array}}
\newcommand{\bad}{\begin{aligned}}
\newcommand{\ead}{\end{aligned}}

\newcommand{\normtwo}[1]{\| #1 \|_2}

\newcommand{\normfro}[1]{\| #1 \|_{\text{F}}}
\newcommand{\normfroinf}[1]{\| #1 \|_{\text{F},\infty}}

\newcommand{\inp}[2]{\left\langle #1, #2 \right\rangle}
\newcommand{\argmin}{\mathop{\rm argmin}}
\newcommand{\argmax}{\mathop{\rm argmax}}

\newcommand{\p}{\mathcal{P}}
\newcommand{\ps}{\mathcal{P_{\St}}}

\newcommand{\dist}{\mathrm{dist}}

\newcommand{\distinfty}{ {d}_{2,\infty}}

\newcommand{\Retr}{\mathrm{Retr}}

\newcommand{\Tr}{\mathrm{Tr}}

\newcommand{\T}{\mathrm{T}}
 
\newcommand{\st}{\mathrm{s.t. }}

\newcommand{\St}{\mathrm{St}}

 \newcommand{\grad}{\mathrm{grad}}
 \newcommand{\Hess}{\mathrm{Hess}}

\newcommand{\Exp}{\mathrm{Exp}}

\newcommand{\blue}{\color{blue}}

\usepackage{textcomp}
\def\BibTeX{{\rm B\kern-.05em{\sc i\kern-.025em b}\kern-.08em
		T\kern-.1667em\lower.7ex\hbox{E}\kern-.125emX}}
\IEEEoverridecommandlockouts                              % This command is only needed if 
\title{
On the Local Linear Rate of Consensus on the Stiefel Manifold  
}

\author{Shixiang Chen$^{1}$, Alfredo Garcia$^{1}$, Mingyi Hong$^{2}$  and Shahin Shahrampour$^{1}$ % <-this % stops a space
\thanks{$^{1}$The Wm Michael Barnes '64 Department of Industrial and Systems Engineering, Texas A\&M University, College Station, TX 77843. 
	Email addresses: {\tt\small sxchen@tamu.edu} (S. Chen), {\tt\small alfredo.garcia@tamu.edu } (A. Garcia), {\tt\small shahin@tamu.edu} (S.  Shahrampour).}%
\thanks{$^{2}$The Department of Electrical and Computer Engineering, University of Minnesota, Minneapolis, MN 55455.
    Email address: {\tt\small mhong@umn.edu} (M. Hong).}%
    }

\begin{document}
\maketitle
%\thispagestyle{empty}  % no numbering on first page
%\pagestyle{empty}  % no numbering all pages

%%%%%%%%%%%%%%%%%%%%%%%%%%%%%%%%%%%%%%%%%%%%%%%%%%%%%%%%%%%%%%%%%%%%%%%%%%%%%%%%
\begin{abstract}
We study the convergence properties of Riemannian gradient method for solving the consensus problem (for an undirected connected graph) over the Stiefel manifold. The Stiefel manifold is a non-convex set and the standard notion of averaging in the Euclidean space does not work for this problem. We propose Distributed Riemannian Consensus on Stiefel Manifold (DRCS) and prove that it enjoys a local linear convergence rate to   global consensus. More importantly, this local rate asymptotically scales with the second largest singular value of the communication matrix, which is on par with the well-known rate in the Euclidean space. To the best of our knowledge, this is the first work showing the equality of the two rates. The main technical challenges include (i) developing a Riemannian restricted secant inequality for convergence analysis, and (ii) to identify the conditions (e.g., suitable step-size and initialization) under which the algorithm always stays in the local region.

%Our result relies upon a generalized Descent Lemma %Lipschitz-type inequality 
%and a restricted secant inequality, we present a sharp analysis of the stepsize and the local linear convergence rate. The linear rate is shown to be asymptotically 

\end{abstract}
 \IEEEpeerreviewmaketitle
 
\section{Introduction}\label{sec:introduction}

Consensus and coordination has been a major topic of interest in the control community for the last three decades. The consensus problem in the Euclidean space is well-studied, but perhaps less well-known is consensus on the Stiefel manifold $\St(d,r):=\{x\in\R^{d\times r}: x^\top x = I_r\}$, which is a non-convex set. This problem has recently attracted significant attention \cite{markdahl2017almost,markdahl2020high,markdahl2018geometric} due to its applications to synchronization in planetary scale sensor networks\cite{paley2009stabilization}, modeling of collective motion in flocks\cite{al2018gradient}, synchronization of quantum bits\cite{lohe2010quantum}, and the Kuramoto models \cite{sarlette2009synchronization,markdahl2020high}. We refer the reader to \cite{markdahl2017almost,markdahl2020high} for more applications of this framework.
 
In general, the optimization problem of consensus on a Riemannian manifold $\M$ can be written as  
\be \label{opt:consensus_prev}
\bad
\min \phi(\X):=  \frac{1}{2}\sum_{i=1}^N\sum_{j=1}^N a_{ij} \mbox{\rm dist}^2(x_i,x_j)\\
\st \quad x_i\in\M,\ i= 1,\ldots,N,
\ead
\ee
where $\mbox{\rm dist}(\cdot,\cdot)$ is a distance function, $a_{ij}\geq 0$ is a constant associated with the underlying undirected, connected graph, and $\X^\top := ( x_1^\top \ x_2^\top \ \ldots \ x_N^\top )$. The consensus problem is also closely related to the center of mass problem on $\M$\cite{afsari2011riemannian}. To achieve consensus, one needs to solve the problem \eqref{opt:consensus_prev} to obtain a global optimal point. The Riemannian gradient method (RGM)\cite{Absil2009,boumal2019global} is a natural choice. When $\M=\St(d,r)$, which is embedded in the Euclidean space, it is more convenient to use the Euclidean distance for both computation and analysis purposes. For example, if the distance function in \eqref{opt:consensus_prev} is the geodesic distance, the Riemannian gradient of $\phi(\X)$  in \eqref{opt:consensus_prev} is the logarithm mapping, which does not have a closed-form solution on $\St(d,r)$ for $1<r<d$, and thus, iterative methods of computing Stiefel logarithm were proposed in \cite{rentmeesters2013algorithms,zimmermann2017matrix}. Moreover, the geodesic distance is not globally smooth.

In this paper, we discuss the convergence of RGM for solving the consensus problem on Stiefel manifold using the square Frobenius norm distance. This problem has been discussed in \cite{sarlette2009consensus, markdahl2020high}, which can be formulated as follows
	\be\label{opt:consensus_thispaper1}\tag{C-St} 
	\bad
	\min  \varphi^t(\X) :=\frac{1}{4} \sum_{i=1}^N\sum_{j=1}^N W^t_{ij} \normfro{x_i-x_j}^2 \\ \st \quad x_i\in\St(d,r),\    i= 1,\ldots,N,
	\ead \ee
where the superscript  $t\geq 1$ is an integer used to denote the $t$-th power of a doubly stochastic matrix $W$. Note that $t$ is introduced here to provide flexibility for our algorithm design and analysis, and computing $W^t_{ij}$ basically corresponds to performing $t$ steps of communication on the tangent space, on which we elaborate in \cref{alg:DRCS}.

It is well-known that for a generic smooth optimization problem over a Riemannian manifold, RGM globally converges to first-order critical points with a sub-linear rate \cite{Absil2009,boumal2019global}. In this  paper, we focus on applying RGM to  \eqref{opt:consensus_thispaper1}, and we call the resulting algorithm Distributed Riemannian Consensus on Stiefel Manifold (DRCS). We prove that for DRCS this sub-linear rate can be improved. In particular, we provide the first analysis showing that, a discrete-time retraction based RGM applied to problem \eqref{opt:consensus_thispaper1} converges  Q-linearly\footnote{A sequence $\{a_k\}$ is said to converge Q-linear to $a$ if there exists $\rho\in(0,1)$ and such that $\lim_{k\rightarrow \infty} \frac{|a_{k+1}-a|}{|a_k-a|} = \rho$. } in a local region  of the global optimal set. Furthermore, we show that the size of the local region and the linear rate are both dependent on  the connectivity of the graph capturing the network structure. Our main technical contributions are as follows: 
\begin{enumerate}
\item We develop and draw upon three second-order approximation properties \eqref{key}-\eqref{ineq:ret_second-order}-\eqref{property of projection onto tangent} in  \cref{lem:distance between Euclideanmean and IAM,lem:nonexpansive_bound_retraction,lem:relation between R and E}, which are crucial to link the Riemannian convergence analysis and Euclidean convergence analysis. 
\item We focus on identifying the suitable stepsize for DRCS, which can guarantee global convergence and local convergence. This is proved by showing a new descent lemma in \cref{lem:lipschitz}. 
\item We will show that a surrogate of local strong convexity holds for problem \eqref{opt:consensus_thispaper1}. It is called the 
\textit{Restricted Secant Inequality} (RSI), derived in \cref{lem:helpful lem to prove X in N_Qt}. In Euclidean space,
RSI  was proposed in \cite{zhang2013gradient} to study the convergence rate for gradient method. The benefit of RSI is that we do not need to take into account the second-order information, and that the linear rate can be proved easily like the Euclidean algorithms. \cref{lem:helpful lem to prove X in N_Qt} can be thought as a Riemannian version of the Euclidean RSI. 
\item Let $\cX^*$ denote the optimal solution set for the problem \eqref{opt:consensus_thispaper1}. It is easy to see that the following holds:
  \be\label{def:opt set} \cX^* : = \{ \X\in\St(d,r)^N: x_1=x_2= \ldots= x_N\}. \ee 
  After establishing the RSI, we prove the local Q-linear consensus rate of $\dist(\X_k,\cX^*)$ for DRCS, where $\dist(\X_k,\cX^*)$ is the Euclidean distance between $\X_k$ and the consensus set $\cX^*$. We show that the convergence rate asymptotically scales with the second largest singular value of $W$, which is the same as its counterpart in the Euclidean space. We characterize two local regions for such convergence in \cref{thm:linear_rate_consensus}, and for the larger region we require multi-step consensus.  
\end{enumerate}

 \subsection{Related Literature} 
  As the general Riemannian manifolds are nonlinear and the problem \eqref{opt:consensus_prev} is non-convex, the  consensus on manifold is considered a more difficult problem than that in the Euclidean space. The first-order critical points are not always in $\cX^*$.  The consensus on Riemannian manifold has been studied in several papers. We can broadly divide their approaches to \textit{intrinsic} or \textit{extrinsic}, which we will describe next.

 The intrinsic approach  means that it relies only on the intrinsic properties of the manifold, such
	as geodesic distances, exponential and logarithm maps, etc. For example, the discrete-time RGM for manifolds with bounded curvature is studied in \cite{tron2012riemannian}. \cite{bonnabel2013stochastic} also studies the stochastic RGM and applies it to solve the consensus problem on the manifold of symmetric positive definite matrix. The authors of \cite{bonnabel2013stochastic} show that using intrinsic approach outperforms the extrinsic method, i.e., the gossip algorithm\cite{boyd2006randomized}.

 The extrinsic approach  is based on specific
	embedding of the manifolds in Euclidean space.
	% The approach in  \cite{sarlette2009consensus} is \textit{extrinsic}.  
	In \cite{sarlette2009consensus}, RGM is also studied for solving the consensus problem over the special orthogonal group $\mathrm{SO}(d)$ and the Grassmannian. However, it is only shown that RGM converges to the critical point. To achieve the global  consensus, a synchronization algorithm on the tangent space is presented in \cite[Section 7]{sarlette2009consensus}. But it requires communicating   an extra variable. 
 
 The main challenge of consensus on manifolds is that the optimization problem  is non-convex. Previous results show that the global consensus is graph dependent, e.g., the global  consensus is achievable  on equally weighted complete graph for $\mathrm{SO}(d)$ and Grassmannian \cite{sarlette2009consensus}. In \cite{tron2012riemannian}, it is also shown that any first-order critical point is the global optima for the tree graph on a manifold with bounded curvature. For general connected undirected graphs, the survey paper \cite{sepulchre2011consensus} summarizes three solutions to achieve almost global consensus on the circle (i.e., $d=2$ and $r=1$):  potential reshaping\cite{sarlette2009synchronization}, the gossip algorithm\cite{sarlette2008global} and dynamic consensus\cite{sarlette2009consensus}. However, such procedures could degrade the convergence speed. For example, the gossip algorithm could be arbitrarily slow and the dynamic consensus is only asymptotically convergent.

 When specific to the Stiefel manifold, most of the previous work for consensus on $\St(d,r)$ is on {\it local} convergence.  For example, the results of \cite{tron2012riemannian} show that, firstly,  any critical point in the region $\mathcal{S}:=\{ \X:\exists y\in\M\  \st \   \max_i d_g(x_i,y)< r^* \}$ is a global optimal point, where $d_g(\cdot,\cdot)$ is the geodesic distance and $r^*$ is an absolute constant with respect to the manifold. Also, the region  $\mathcal{S}$ is   convex\footnote{An open subset $s\subset \M $ is convex if it contains all shortest paths between any two points of $s$. }. %\cite[Definition 2.3]{afsari2013convergence}.  
  Secondly, RGM is  shown to achieve consensus locally.  Specifically, if the initial point $\X_0$ satisfies $\X_0\in \mathcal{S}_{\text{conv}}:=\{ \phi(\X)< \frac{(r^*)^2}{2dia(\G)} \}$, where $ dia(\G)$ is the diameter of the graph $\G$, then RGM  converges to global optimal point. However, the region $\mathcal{S}_{\text{conv}}$ is much smaller compared with $\mathcal{S}$ since   $\X\in\mathcal{S}_{\text{conv}}$ implies that $\sum_{j=1}^N a_{ij}d_g^2(x_i,x_j)\leq 2\phi(\X) \leq (r^*)^2/dia(\G)$. The difficulty of showing the consensus region to be $\mathcal{S}$ lies in preserving the iterates in $\mathcal{S}$. To theoretically guarantee this, the  sectional curvature of the manifold should be constant and non-negative, e.g., the sphere, or when the graph $\G$ has a linear structure.   
  
 Recently, the authors of  \cite{markdahl2017almost,markdahl2020high} show that one can achieve almost global consensus for problem \eqref{opt:consensus_thispaper1} whenever $r\leq \frac{2}{3}d-1$. More specifically, all second-order critical points are global optima, and thus, the measure of  stable manifold of saddle points is zero.  This can be proved by showing that the Riemannian Hessian at all saddle points has negative curvature, i.e., the strict saddle property in \cite{lee2017first} holds true.  Therefore, if we randomly initialize the RGM, it will almost always converge to the global optimal point \cite{lee2017first,markdahl2020high}.  Additionally, \cite{markdahl2020high} also conjectures that the strict saddle property holds for $d\geq 3$ and $r\leq d-2$. The scenarios $r=d-1$ and $r=d$ correspond to the multiply connected ($\St(d,d-1)\cong \mathrm{SO}(d)$)  and not connected case ($\St(d,d)\cong \mathrm{O}(d)$), respectively, which yields  multi-stable systems\cite{markdahl2019synchronization}.      %{\color{red}[why is it necessary to mention the last sentence?]{\blue [SX: Just want to mention why it doesn't hold for r=d-1 and r=d]} [need to mention the drawback of the above work, i.e., no rate analysis is given.]} 

{However, none of the aforementioned work discusses the local linear rate of RGM on $\St(d,r)$ with $r>1$. }
One way  to prove the linear rate is to show that the Riemannian Hessian is positive definite  \cite{Absil2009} near a consensus point, but the Riemannian Hessian is degenerate at all consensus points (see \cref{sec:linear rate}). The linear rate of consensus can be established   by  reparameterization on the circle \cite{sarlette2009synchronization} or computing the generalized Lyapunov-type numbers on the sphere\cite{lageman2016consensus}, but it is not known how to generalize them to $r>1$. Thanks to the recent advancements in non-convex optimization\cite{lee2017first,ma2019implicit,Boumal-phase-synchronization-2016} and optimization over Stiefel manifold \cite{edelman1998geometry,Abrudan2008,Absil2009,boumal2019global,Liu-So-Wu-2018,chen2020proximal,li2019nonsmooth}, we  study the local landscape of \eqref{opt:consensus_thispaper1} by an extrinsic approach and tackle the problem using a Riemannian-type RSI.

\section{Preliminaries}

\subsection{Outline of the Paper and Notation}
The rest of the  paper is organized as follows. \cref{sec:alg} describes the algorithm and challenges. \cref{sec:consensus} presents the global convergence results. \cref{sec:linear rate} develops the Riemannian RSI and the local linear rate. \cref{sec:numerical} demonstrates the numerical experiments. \cref{sec:appendix}  provides the proofs of all technical results. 

Starting from this section, we use $\M =\St(d,r) $ for brevity. We also have the following notations: 
\begin{itemize}
	\item  $\G=(\mathcal{V},\mathcal{E})$: the undirected graph with $|\mathcal{V}|=N$ nodes.
	\item $A=[a_{ij}]$: the adjacency matrix of graph $\G$.
	\item   $\X$: the collection of all local variables $x_i$ by stacking them, i.e., $\X^\top = ( x_1^\top \ x_2^\top \ \ldots \ x_N^\top )$. 
	\item $\M^N = \M\times\ldots\times\M$:  the $N-$fold Cartesian product.
	\item  $[N]: = \{1,2,\ldots,N\}$. For $\X\in({\R^{d\times r}})^N$,  the $i$-th block of $\X$: $[\X]_i=x_i$. 
	\item $\nabla \varphi^t(\X)$: Euclidean gradient;  $\nabla \varphi^t_i(\X): = [\nabla \varphi^t(\X)]_i$: the $i$-th block of  $\nabla \varphi^t(\X)$.
		\item $\T_x\M$: the tangent space of $\St(d,r)$ at point $x$.
	\item $N_x\M$: the normal space of $\St(d,r)$ at point $x$.
	\item  $\Tr(\cdot)$: the trace;  $\inp{x}{y}=\Tr(x^\top y)$ : the inner product on $\T_x\M$ is induced from the Euclidean inner product.
	\item $\grad \varphi^t(\X)$: Riemannian gradient;  $\grad \varphi_i^t(\X):= [\grad \varphi^t(\X)]_i$:  the i-th block of  $\grad \varphi^t(\X)$.
	\item  $\normfro{\cdot}$: the Frobenius norm;  $\normtwo{\cdot}$: the operator norm.
	\item $\p_{C}$:  the orthogonal projection onto a closed set $C$.
	\item  $I_r$: the $r\times r$ identity matrix.
	\item $\textbf{1}_N\in\R^N$: the vector of all ones; $J := \frac{1}{N} \textbf{1}_N\textbf{1}_N^\top $.
\end{itemize}

\begin{definition}[Consensus]\label{def:consensus}
	Consensus is the configuration where $x_i = x_j\in\M$ for all $i,j\in[N]$. 
\end{definition}

\subsection{Network Setting}
To represent the network, we use a graph $\G$ that satisfies the following assumption.
\begin{assumption}\label{assump:doubly-stochastic}
	We assume that the undirected graph $\G$ is connected and the corresponding communication matrix $W$ is doubly stochastic, i.e.,
	\begin{itemize}
		\item $W=W^\top$.
		\item $W_{ij}\geq 0$ and $1>W_{ii}>0$.
		\item Eigenvalues of $W$ lie in $(-1,1]$. The second largest singular value $\sigma_2$ of $W$ lies in $[0,1)$.
	\end{itemize}
\end{assumption}
It is easy to see that any power of the matrix $W$ is also doubly stochastic and symmetric. Moreover, the second largest singular value of $W^t$ is $\sigma_2^t$. 
\subsection{Optimality Condition }
We first introduce some preliminaries about optimization on a Riemannian manifold. Let us   consider the following   optimization problem over a matrix manifold $\M$ \be\label{prob:central} \min f(x) \quad \st \quad  x\in\M. \ee
The Riemannian gradient $\grad f(x)$ is defined by the unique tangent vector  satisfying $\inp{\grad f(x)}{\xi} = Df(x)[\xi]$ for all $\xi\in\T_x\M$, where $D$ means the differential of $f$ and  $Df(x)[\xi]$  means the directional derivative along $\xi$.  Since we use the metric on the tangent space $\T_x\M$ induced from the Euclidean inner product $\inp{\cdot}{\cdot}$, the Riemannian gradient $\grad f(x)$ on $\St(d,r)$ is given by $\grad f(x) = \p_{\T_x\M}(\nabla f(x))$, where $\p_{\T_x\M}$ is the orthogonal projection onto $\T_x\M$. More specifically, we have \[ \p_{\T_x\M}(y) = y - \frac{1}{2} x(x^\top y + y^\top x),\] for any $y\in\R^{d\times r}$ (see \cite{edelman1998geometry,Absil2009}), and \[\p_{N_{x}\M}(y) = \frac{1}{2} x(x^\top y + y^\top x).\]  Under the Euclidean metric, the Riemannian Hessian denoted by $\Hess f(x)$ is given by $\Hess f(x)[\xi] = \p_{\T_x\M}( D(x\mapsto \p_{\T_x\M}\nabla f(x)) [\xi] ) $ for any $\xi\in\T_x\M$, i.e., the projection   differential of the  Riemannian gradient\cite{Absil2009,boumal2019global}.  We refer to \cite{absil2013extrinsic} for how to compute $\p_{\T_x\M}( D(x\mapsto \p_{\T_x\M}\nabla f(x)) [\xi] )$ on $\St(d,r)$. The  necessary optimality condition of problem \eqref{prob:central} is given as follows.  
\begin{proposition}(\cite{Yang-manifold-optimality-2014,boumal2019global})\label{prop:optcond}
	Let $x \in \M $ be a local optimum for \eqref{prob:central}. If $f$
	is differentiable at $x$, then $\grad f(x) = 0$. Furthermore, if f is twice differentiable at $x$, then $\Hess   f(x) \succcurlyeq 0$.
\end{proposition} 
A point $x$ is a first-order critical point (or critical point) if $\grad f(x)=0$. $x$ is called a second-order critical point if $\grad f(x)=0$ and  $\Hess   f(x) \succcurlyeq 0$.\\
 The concept of a retraction \cite{Absil2009}, which is a first-order approximation of the exponential mapping and can be more amenable to computation, is given as follows.
%An important concept in manifold optimization is the retraction operation, which is defined as follows.
%Once we get a direction $V_k$ in the tangent space, $X_k+V_k$ is not well-defined on $S_t(n,p)$. The notion of \textbf{retraction} refers to  mapping $X_k+V_k$ onto $S_t(n,p).$
\begin{definition}\label{def_retraction}\cite[Definition 4.1.1]{Absil2009}
A retraction on a differentiable manifold $\mathcal{M}$ is a smooth mapping $\Retr$ from the tangent bundle $\T\mathcal{M}$ onto $\mathcal{M}$ satisfying the following two conditions (here $\Retr_x$ denotes the restriction of $\Retr$ onto $\T_x \mathcal{M}$):
\begin{enumerate}
\item $\Retr_x(0)=x, \forall x\in\M$, where $0$ denotes the zero element of $\T_x\mathcal{M}$.
\item For any $x\in\M$, it holds that
    \[\lim_{\T_x\M\ni\xi\rightarrow 0}\frac{\|\Retr_x(\xi)-(x+\xi)\|_F}{\|\xi\|_F} = 0.\]
%\item $DR_X(0)=Id_{T_X\mathcal{M}}$ by identifying $T_0(T_X\M)=T_X\M$.
\end{enumerate}
\end{definition}
\section{The Proposed Algorithm}\label{sec:alg}
 
The discrete-time RGM applied to solve problem \eqref{opt:consensus_thispaper1}  is described in \cref{alg:DRCS}. We name it as Distributed Riemannian Consensus on Stiefel manifold (DRCS). The goal of this paper is to study the local (Q-linear) rate of  DRCS for solving problem \eqref{opt:consensus_thispaper1}.
\begin{algorithm}[ht]
	\caption{Distributed Riemannian Consensus on Stiefel manifold (DRCS)}\label{alg:DRCS}
	\begin{algorithmic}[1]
		\State{Input: random initial point $\X_0 \in \St(d,r)^N$, stepsize $0<\alpha < 2/L_t $ and an integer $t\geq  1$. }  
		\For{$k=0,1,\ldots$}\Comment{For each node $i\in[N]$, in parallel  }
		\State Compute $\nabla \varphi_i^1(\X_k) = x_{i,k} - \sum_{j=1}^N W_{ij}x_{j,k}$.
		\For{$l=2,\ldots,t$}\Comment{	Multi-step consensus}
		\State{ $\nabla \varphi_i^{l}(\X_k) = {\blue \nabla \varphi_i^1(\X_k)+ } \sum_{j=1}^N W_{ij}\nabla \varphi_j^{l-1}(\X_k) $ }
		\EndFor
		\State{Update \be\label{consensus_rga}x_{i,k+1} = \Retr_{x_{i,k}}\left( -\alpha \p_{\T_{x_{i,k}}\M} \left(\nabla \varphi_i^{t}(\X_k) \right) \right)\ee}
%		\State{Update 
		\EndFor
	\end{algorithmic}
\end{algorithm}
%To achieve   consensus configuration   in Definition \ref{def:consensus}, the initialization should be in the consensus region $\N$.  This may make \cref{alg:DRCS} impractical for consensus problem in some scenarios, but studying consensus problem can help   study the decentralized gradient method for \eqref{opt_problem}.

We remark that the DRCS algorithm is similar in spirit to the Riemannian consensus algorithm in \cite{tron2012riemannian}, but we use   retraction instead of the exponential map. In \cite{tron2012riemannian}, geodesic distance is used in \eqref{opt:consensus_prev} for Grassmannian manifold and special orthogonal group and only a sub-linear rate was shown (using one-step communication). Given some integer $t\geq 1$, the iteration \eqref{consensus_rga} in Algorithm \ref{alg:DRCS} is the Riemannian gradient descent step,
where  $\alpha$ is the stepsize. The algorithm  updates along a negative Riemannian gradient direction on the tangent space, then performs the {retraction} operation $ \Retr_{\X_k}$ to guarantee feasibility.

Also notice that $\normfro{x}^2=r$ holds true for any $x\in\St(d,r)$, so \eqref{opt:consensus_thispaper1} is equivalent to 
\be\label{opt:consensus_thispaper}
\bad
\max  h^t(\X):= \frac{1}{2}\sum_{i=1}^N\sum_{j=1}^N  W^t_{ij}\inp{x_i}{x_j}\\   \st \quad x_i\in\St(d,r),\  \forall i\in[N]. 
\ead  
\ee 
DRCS can also be  seen as applying Riemannian gradient ascent to solve \eqref{opt:consensus_thispaper}. That is, \eqref{consensus_rga} is equivalent to  \be\label{consensus_rga-1}x_{i,k+1} = \Retr_{x_{i,k}}\left( \alpha\p_{\T_{x_{i}}\M} (\sum_{j=1}^N  W_{ij}^t x_{j,k}) \right).\ee    The term $\p_{\T_{x_{i}}\M} (\sum_{j=1}^N W_{ij}^t x_{j,k}) $ can be   viewed as performing $t$ steps of Euclidean consensus  on the tangent space $\T_{x_{i,k}}\M$. 

%where the matrix $W^t$ means the $t-$th power of $W$. Therefore, we are using $W^t$ instead of $W$ in problem \eqref{opt:consensus_thispaper}.
Although multi-step consensus requires more communications at each iteration, it reduces the outer loop iteration number since $\sigma_2^t$ scales  better than $\sigma_2$. %This is because $\sigma_2^t=\frac{1}{2\sqrt{N}}$. 
For a large $t$, the corresponding graph of $W^t\approx\frac{1}{N}\mathbf{1}_N\mathbf{1}_N^\top$ is approximately the complete graph.  %According to \cite{sarlette2009consensus}, the only critical point is the global optimal point on the special orthogonal group and Grassmannian.   
We emphasize here that  multi-step consensus does not make the convergence analysis trivial, since we do not require $t$ to be too large.   For the Euclidean case, \cite{berahas2018balancing} also discusses the   advantages  of multi-step consensus for decentralized gradient method.

\subsection{Consensus in Euclidean Space: A Revisit}\label{sec:euclidean-consensus}
Let us briefly review the consensus with convex constraint in the Euclidean space (C-E)\cite{Tsitsiklisthesis}, which will give us some insights to study the convergence rate of DRCS. The optimization problem can be written as follows
	\be\label{opt:Euclidean consensus}\tag{C-E} 
\bad
\min  \varphi^t(\X) :=\frac{1}{4} \sum_{i=1}^N\sum_{j=1}^N W^t_{ij} \normfro{x_i-x_j}^2 \\ \st \quad x_i\in \mathcal{C},\    i= 1,\ldots,N,
\ead \ee
where $\mathcal{C}$ is a closed convex set in the Euclidean space. Then, the iteration is given by \cite{nedic2010constrained}
\[  x_{i,k+1} = \p_{\mathcal{C}} \left( \sum_{j=1}^N W_{ij} x_{i,k} \right)  \quad \forall i \in[N],  \]
with the corresponding matrix form being as follows  
\be\label{Euclidean consensus}  \X_{k+1} =\p_{\mathcal{C}^N} \left( (W\otimes I_d)  \X_k \right), \tag{EuC} \ee
where $\mathcal{C}^N=\mathcal{C}\times\cdots\times \mathcal{C}.$
Different forms of \eqref{Euclidean consensus} are discussed in \cite{nedic2018network}. 
%	For the  iteration \eqref{Euclidean consensus}, we have the following properties:
%	\begin{enumerate}
%		\item $  \hat x_0 = \frac{1}{N}\sum_{i=1}^{N} x_{i,0} = \hat x_1 =\ldots = \hat x_k$ for any $k\geq 0$;
%		\item $\normfro{\X_k - \hat\X_k} \leq \sigma_2  \normfro{\X_{k-1} - \hat\X_{k-1}}, $ 
%	\end{enumerate}
%	where $\sigma_2$ is the second largest singular value of $W$.
%	The first property comes from the doubly stochasticity of $W$ and the second can be derived as follows
Let us denote the Euclidean mean via 
\be\label{def:Euclidean mean }\hat x: = \frac{1}{N}\sum_{i=1}^{N} x_i \ \text{and} \  \hat \X := \mathbf{1}_N \otimes\hat x .\ee
We have 
	\be\label{linear rate Euclidean}
	\begin{aligned}
		\normfro{\X_k - \hat\X_k} 
		&\leq\normfro{\X_k - \hat\X_{k-1}} \\
		&= \normfro{ \p_{\mathcal{C}^N} \left( (W\otimes I_d )  \X_{k-1} \right) - \hat\X_{k-1}} \\
		& \leq  \normfro{ [(W  - J )\otimes I_d]  ( \X_{k-1} - \hat\X_{k-1})}\\
		&\leq \sigma_2  \normfro{\X_{k-1} - \hat\X_{k-1}},
	\end{aligned}\ee
where the second inequality follows from the non-expansiveness of $\p_{\mathcal{C}}$. 
	Therefore, the Q-linear rate of \eqref{Euclidean consensus} is equal to $\sigma_2$.
	On the other hand, the iteration \eqref{Euclidean consensus} is the same as applying projected gradient descent (PGD) method to solve the problem \eqref{opt:Euclidean consensus}.  
	That is, we have
	\be\label{e}
	\X_{k+1} =\p_{\mathcal{C}^N} \left( (W\otimes I_d)  \X_k \right) =  \p_{\mathcal{C}^N} \left( \X_k -   \alpha_{e}\nabla \varphi(\X_k) \right),
	\ee
	with stepsize $\alpha_{e}= 1$. 
	Let us take a look at how to show the linear rate of PGD using standard convex optimization analysis.  We have the Euclidean gradient $\nabla \varphi(\X) = \X - (W\otimes I_d) \X$. Though the hessian matrix $\nabla^2 \varphi(\X)=(I_N - W)\otimes I_d$ is degenerated, it is positive definite when restricted to the subspace  $(\R^{d\times r})^N   \setminus \mathcal{E}^*$, where $\mathcal{E}^* : = \mathbf{1}_{N} \otimes \R^{d\times r}  $ is the optimal set of CE problem. Simply speaking, $I_N-W$ is positive definite in $\R^{N}\setminus \mathrm{span}(\mathbf{1}_N)$.  Note that $\hat\X = \p_{ \mathcal{E}^* } \X$, so $\X - \hat \X$ is orthogonal to $\mathcal{E}^* $.  
	Following the proof of linear rate for strongly convex functions\cite[Theorem 2.1.15]{nesterov2013introductory}, one needs the inequality in \cite[Theorem 2.1.12]{nesterov2013introductory}, specialized to our problem as follows 
	\be\label{restricted strong convexity}
	\bad
	& \quad \inp{\X - \hat\X}{\nabla \varphi (\X)}\\
	& =\inp{\X - \hat\X}{(I_N - W)\otimes I_d (\X - \hat\X)} \\
	&\geq \frac{\mu L}{\mu + L}\normfro{\X - \hat\X}^2 + \frac{1}{\mu + L} \normfro{\nabla \varphi(\X)}^2.
		\ead \ee
	The constants are given by \[ \mu := 1-\lambda_{2}(W)\quad \text{and} \quad L : = 1-\lambda_{N}(W), \] where $\lambda_{2}(W)$ is the second largest eigenvalue of $W$, and $\lambda_{N}(W)$ is the smallest eigenvalue of $W$, respectively. This inequality can be   obtained using the eigenvalue decomposition of $I_N-W$. We provide the proof in the Appendix, and we call \eqref{restricted strong convexity}  ``restricted secant inequality''. 
	With this, if $\alpha_{e} = \frac{2}{\mu + L }$, we get 
	\[ \normfro{\X_k - \hat\X_k} \leq (\frac{L  - \mu}{L + \mu})^k \normfro{\X_0 - \hat \X_0}. \]
	It can be shown by simple calculations that $\frac{L- \mu }{L+\mu}\leq \sigma_2$. This suggests that the  PGD can achieve faster convergence rate with $\alpha_{e} = \frac{2}{\mu + L }$. When $\alpha_{e} = 1$, the rate of $\sigma_2$ can be shown via combining  \eqref{restricted strong convexity} with $L\normfro{\X - \hat \X} \geq \normfro{\nabla \varphi(\X)} \geq \mu \normfro{\X - \hat \X}$. The proof is provided in the Appendix. 

%\subsection{Basic Lemmas}\label{subsec:basic lemmas}
\subsection{Consensus on Stiefel Manifold: Challenges and Insights }\label{sec:consensus on stiefel}
As we see, different from the \eqref{Euclidean consensus} iteration with convex constraint\cite{nedic2010constrained}, in DRCS the projection onto convex set is replaced with a retraction operator, and the Euclidean gradient is substituted by the Riemannian gradient. The standard results \cite{Absil2009,boumal2019global} on RGM already show global sub-linear rate of DRCS. However, to obtain the {\it local Q-linear} rate, we need to exploit the specific problem structure.  To analyze DRCS, there are two main challenges. 

   First, due to the non-linearity of $\St(d,r)$, the Euclidean mean $\hat x$ in \eqref{def:Euclidean mean } is infeasible. We need to use the average point defined on the manifold. The second challenge comes from the non-convexity of $\St(d,r)$. Previous work  such as \cite{nedic2010constrained} usually discusses  the convex constraint in the Euclidean space, which depends on the non-expansive property of the projection operator onto convex constraint. 
   
   To solve these issues. We use the so-called induced arithmetic mean (IAM)   \cite{sarlette2009consensus}   of $x_1,\ldots,x_N$ over $\St(d,r)$, defined by
\begin{align} 
\bar{x}& := \argmin_{y\in\St(d,r)} \sum_{i=1}^N\normfro{y-x_{i}}^2\notag\\
%&=\argmax_{y\in\St(d,r)} \inp{y}{\sum_{i=1}^N x_{i}}  = \ps(\hat x),\label{eq:IAM}\tag{IAM}
&=\argmax_{y\in\St(d,r)} \langle y,\sum_{i=1}^N x_{i}\rangle  = \ps(\hat x),\label{eq:IAM}\tag{IAM}
\end{align}
where $\p_{\St}(\cdot)$ is the orthogonal projection onto $\St(d,r)$. Different from the Euclidean mean notation, we define
\be\label{def:IAM}
\bar{x}_k= \ps(\hat x_k)  \quad \text{and}\quad \bar\X_k = \mathbf{1}_N\otimes \bar x_k 
\ee
to denote IAM of $x_{1,k},\ldots,x_{N,k}$.
The IAM  is the orthogonal projection of the Euclidean mean onto $\St(d,r)$, and $\bar\X$ is also the projection of $\X$ onto the optimal set $\cX^*$ defined in \eqref{def:opt set}. The distance between $\X$ and $\cX^*$  is given by
\[\dist^2(\X,\cX^*) =  \min_{y\in\St(d,r)} \frac{1}{N} \sum_{i=1}^N\normfro{y-x_{i}}^2 = \frac{1}{N}\normfro{\X-\bar\X}^2.  \]%Note that the projection can be efficiently obtained via the compact singular value decomposition(SVD), whose complexity of a $d\times r$ matrix is $\mathcal{O}(dr^2)$.
 The terminology IAM is derived from \cite{moakher2002means}, where the IAM on $\mathrm{SO}(3)$ is called the projected arithmetic mean. The IAM is different from  the Fr\'{e}chet mean \cite{afsari2011riemannian,afsari2013convergence,tron2012riemannian} (or the Karcher mean\cite{grove1973conjugatec,karcher1977riemannian}). %,  which is defined by the geodesic distance $d_g$ on $\M$ as follows:
%\[ \bar{ x}_{\text{ Fr\'{e}chet}} = \argmin_{y\in\St(d,r)} \sum_{i=1}d_g^2(y,x_{i}). \] 
We use IAM since it is easier to adopt to the Euclidean linear structure and computationally convenient.   
Furthermore, we define the $l_{F,\infty}$ distance between $\X_k$ and $\bar \X_k$ as 
%\be\label{distance_two}
%\disttwo^2(\X,\bar\X_k) = \sum_{i=1}^N\normfro{x_i-\bar{x}}^2,
%\ee
\be\label{distance_infty}
\normfroinf{\X-\bar\X} = \max_{i\in[N]} \normfro{x_i-\bar{x}}. \tag{$l_{F,\infty}$}
\ee

Let us first build the connection between the Euclidean mean and IAM in the following lemma. 
\begin{lemma}\label{lem:distance between Euclideanmean and IAM}
	For any $\X\in\St(d,r)^N$, let $\hat x  =\frac{1}{N}\sum_{i=1}^N x_i$ be the Euclidean mean and denote $\hat\X = \mathbf{1}_N\otimes \hat x$ defined in \eqref{def:Euclidean mean }. Similarly, let $\bar\X = \mathbf{1}_N\otimes \bar x$, where $\bar{x}$ is the IAM defined in \eqref{eq:IAM}.  We have
	\be\label{ineq:relation_eum_manifmean} \frac{1}{2}\normfro{\X- \bar\X}^2\leq \normfro{\X - \hat \X}^2\leq \normfro{\X- \bar\X}^2.  \ee
	Moreover, if $\normfro{\X - \bar\X}^2\leq N/2$, one has  
	\begin{equation}\label{key}
	\normfro{\bar x - \hat x }\leq \frac{2\sqrt{r}\normfro{\X-\bar\X}^2}{N}, \tag{P1}
	\end{equation}
	and 
	\be\label{ineq:relation_eum_manifmean_1}
	\normfro{\X - \hat\X}^2 \geq \normfro{\X-\bar\X}^2  -  \frac{4r \normfro{ \X-\bar\X}^4}{ {N}}.
	\ee
\end{lemma}
The inequality \eqref{ineq:relation_eum_manifmean} is tight, since we have $\frac{1}{2}\normfro{\X- \bar\X}^2= \normfro{\X - \hat \X}^2=Nr$ when $\sum_{i=1}^N x_i = 0$ and $\normfro{\X - \hat \X}^2= \normfro{\X- \bar\X}^2$ when $x_1=x_2=\ldots=x_N$. 
  The inequality \eqref{key} suggests that the Euclidean mean will converge to IAM quadratically if $\X$ is close to $\bar \X$.  

 To deal with the non-convexity of $\St(d,r$),  we use the nice properties for second-order retraction. 
%Although there is no linear structure on $\St(d,r)$,
The following second-order property of retraction in \cref{lem:nonexpansive_bound_retraction} is crucial to link the optimization methods between Euclidean space  and  the matrix manifold.  It means that $\Retr_{x}(\xi) = x+\xi + \mathcal{O}(\normfro{\xi}^2)$, that is, $\Retr_{x}(\xi)$ is locally a good approximation to $x+\xi$. This property has been used to analyze many algorithms (see e.g., \cite{boumal2019global, chen2020proximal, li2019nonsmooth}).  In this paper, we only use the polar decomposition based retraction to present a simple proof. The polar decomposition is given by 
\begin{align}\label{eq:polar}
    \Retr_{x}(\xi) = (x+\xi)(I_r + \xi^\top \xi)^{-1/2},
\end{align}
which is also the orthogonal projection of $x+\xi$ onto $\St(d,r)$.   The following property \eqref{ineq:ret_nonexpansive} also holds for the polar   retraction, which can be seen as a non-expansiveness property. 
%Commonly used retractions on Stiefel manifold include   polar decomposition\cite{Absil2009}, QR-based retraction\cite{Absil2009}, Caley transformation\cite{Wen-Yin-2013} and some  other variants\cite{jiang2015framework}. The   exponential mapping\cite{edelman1998geometry} is definitely   a retraction, yet we use retraction to distinguish from exponential map.  The computation complexity of these retractions and exponential on $\St(d,r)$ are all $\mathcal{O}(dr^2+r^3)$. But the constants of retractions are smaller.  
\begin{lemma} \cite{boumal2019global,Liu-So-Wu-2018}\label{lem:nonexpansive_bound_retraction}
	Let $\Retr$ be a {second-order} retraction over $\St(d,r)$. We then have
	\be\label{ineq:ret_second-order} \bad
	\normfro{ \Retr_x(\xi) - (x+\xi)}\leq M\normfro{\xi}^2, \\
	\forall x\in\St(d,r), \quad \forall \xi\in\T_x\M.  
	\ead \tag{P2}
	\ee
	Moreover, if the retraction is the polar retraction, then  
	for all $\X\in\St(d,r)$ and
	$\xi\in\T_x\M$, the following  inequality  holds for any $y\in \St(d,r)$ \cite[Lemma 1]{li2019nonsmooth}:
	%	\be\label{ineq:ret_first_bound}
	%	\normfro{\Retr_{\X}(\xi)-\X}\leq M_1\normfro{\xi},  
	%	\ee
	\be\label{ineq:ret_nonexpansive}
	\normfro{\Retr_{x}(\xi)-y}\leq \normfro{x+\xi-y} .
	\ee
\end{lemma}
\begin{remark}\label{remark:boundness on retraction}
	The constant $M$ in \eqref{ineq:ret_second-order} depends on the retraction. \cite{boumal2019global} established \eqref{ineq:ret_second-order} for all $\xi$.  If $\xi$ is uniformly  bounded\cite{Liu-So-Wu-2018}, then we   have a constant bound for $M$, which is independent of the dimension. For example, \cite[Append. E]{Liu-So-Wu-2018}  shows that if $\normfro{\xi}\leq 1$ then $M=1$   for polar retraction.  If $\normfro{\xi}\leq 1/2$ then $M=\sqrt{10}/4$  for QR decomposition\cite{Absil2009} and if  $\normfro{\xi}\leq 1/2$ then $M=4$ for Caley transformation\cite{Wen-Yin-2013}. The uniform bound of $\normfro{\xi}\leq 1$ will be satisfied automatically under mild assumptions. We remark that the inequality \eqref{ineq:ret_nonexpansive} will help with simplifying some of our analysis. If we do not use polar retraction, using \eqref{ineq:ret_second-order} implies 
	\be\label{ineq:ret_nonexpansive-1}	\normfro{\Retr_{x}(\xi)-y}\leq \normfro{x+\xi-y} + M\normfro{\xi}^2,\ee
	where the second-order term $M\normfro{\xi}^2$ changes the suitable step-size range in most of our analysis. 
\end{remark}

 We now show the relation between  $\nabla\varphi^t (\X)$ and $\grad \varphi^t(\X)$. Denoting $\p_{N_{x}\M}$ as the orthogonal projection onto the normal space $N_x\M$, a useful property of the projection  $	\p_{\T_x\M}(y - x), \forall y\in\St(d,r)$\cite[Section 6]{li2019nonsmooth} is that 
	\be\label{property of projection onto tangent}
	\bad
  &\quad  \p_{\T_x\M}(x - y) = x -  y - \p_{N_{x}\M}(x-y) \\
%  y - \frac{1}{2} x(x^\top y + y^\top x)\\
 	& = x - y - \half x\left( (x-y)^\top x + x^\top (x-y) \right)\\
& = x - y - \half x (x-y)^\top (x-y),
	\ead  \tag{P3}
  \ee
where we used $x^\top x = y^\top y = I_r$. This property implies that
\[  \p_{\T_x\M}(x - y) = x - y + \mathcal{O}(\normfro{y-x}^2). \] 
%That is, the projected distance between $x$ and $y$ onto $\T_{x}\M$ is the second-order approximation of the Euclidean distance. 
The relationship \eqref{property of projection onto tangent} implies the following lemma.

\begin{lemma}\label{lem:relation between R and E}
	For any $\X,\Y\in\St(d,r)^N$, we have	
	\be\label{ineq:key relation between R and E}
	\bad
	& \quad  \inp{\grad \varphi^t(\X)}{\Y-\X} 
	  =   \inp{\nabla \varphi^t(\X)}{\Y-\X } +\\
	 &  \frac{1}{4}\sum_{i=1}^N   \langle \sum_{j=1}^N W^t_{ij} (x_i - x_j)^\top (x_i - x_j),(y_i-x_i)^\top (y_i-x_i) \rangle\\
	 &\geq 	  \inp{\nabla \varphi^t(\X)}{\Y-\X}.
	 \ead
	 \ee
\end{lemma}

\cref{lem:relation between R and E} directly yields   a descent lemma on the Stiefel manifold similar to the Euclidean-type inequality \cite{nesterov2013introductory}, which is helpful to identify the stepsize for global convergence. The stepsize $\alpha$ will be determined by the constant $L_t$ in \cref{lem:lipschitz} and the constant $M$ in \cref{lem:nonexpansive_bound_retraction}.   Lemma \ref{lem:lipschitz} is developed from a so-called Riemannian inequality in \cite{li2019nonsmooth}, which is used to analyze a class of Riemannian subgradient methods.  For the  function $\varphi^t(\X)$, we get a tighter estimation of $L_t$.
\begin{lemma}[Descent lemma] \label{lem:lipschitz}
	For the function $\varphi^t(\X)$ defined in \eqref{opt:consensus_thispaper1},
	 we have 
	\be\label{ineq:lipschitz}
	\bad
	&\quad \varphi^t(\Y) - \left[ \varphi^t(\X) + \inp{\grad \varphi^t(\X)}{\Y-\X} \right]  \\
	&\leq \frac{L_t}{2}\normfro{\Y-\X}^2, \quad \forall  \X,\Y\in\St(d,r)^N,
	\ead\ee
	where $L_t =   1 - \lambda_{N}(W^t) $ and $\lambda_{N}(W)$ is the smallest eigenvalue of $W$.
%	Therefore, one also has 
%	\be\label{ineq:lips_riemanniangrad} \inp{\grad  \varphi^t(\X) -\grad  \varphi^t(\Y)}{\X-\Y}\leq L\normfro{\Y-\X}^2. \ee
%	Furthermore, given $\X,\Y\in\St(d,r)^N$, if $\frac{1}{2} \sum_{j=1}^N W^t_{ij} [ x_j^\top x_i + x_i^\top x_j ]$ and  $\frac{1}{2} \sum_{j=1}^N W^t_{ij} [ y_j^\top y_i + y_i^\top y_j ]$ are positive semidefinite for all $i\in[N]$, it follows that
%	\be\label{ineq:lipschitz-local}\bad
%	\varphi^t(\Y) - \left[ \varphi^t(\X) + \inp{\grad  \varphi^t(\X)}{\Y-\X} \right]  \leq \half \normfro{\Y-\X}^2.
%	\ead\ee	 
%	and \be\label{ineq:lips_riemanniangrad-local} \inp{\grad  \varphi^t(\X) -\grad  \varphi^t(\Y)}{\X-\Y}\leq \normfro{\Y-\X}^2. \ee
\end{lemma}

%The difference between two Riemannian gradients is not well-defined on general manifold. However, since the Stiefel manifold is embedded in Euclidean space, we are free to do so。 As such, it would be more mathematically sound to view $\grad \varphi^t(\X)$ as the projected Euclidean gradient.  
We remark that a closely related inequality is the  restricted Lipschitz-type gradient presented in \cite[Lemma 4]{boumal2019global}, which is defined by the pull back function $g(\xi):=\varphi^t(\Retr_{\X}(\xi))$, whose Lipschitz $\tilde{L}$ relies on the retraction and the Lipschitz constant of Euclidean gradient. Also, the stepsize of RGM in \cite{boumal2019global} depends on the norm of Euclidean gradient.  Our inequality  does not rely on the retraction, which could be of independent interest.  
One could also consider the following Lipschitz  inequality (e.g., see \cite{zhang2016first})
\be
\begin{aligned}
\varphi^t(\Y)  \leq  \varphi^t(\X) + \inp{\grad \varphi^t(\X)}{\Exp_{\X}^{-1}\Y} + \frac{L_g}{2}d_g^2(\X,\Y)
\end{aligned}
\ee
%\[ \normfro{\mathrm{P}_{\X\rightarrow \Y}\grad \varphi^t(\X) - \grad \varphi^t(\Y)}\leq L_g^\prime d_g(\X,\Y), \]
%where $\mathrm{P}_{\X\rightarrow \Y}:\T_\X\M \rightarrow \T_\Y\M$ is the parallel transport 
where $\Exp_{\X}^{-1}\Y$ is the logarithm map and $d_g(\X,\Y)$ is the geodesic distance. Since involving logarithm map and geodesic distance brings computational and conceptual difficulties, we choose to use the form of \eqref{ineq:lipschitz}  for simplicity.  In fact, $L_t$   and $L_g$ are  the same for problem \eqref{opt:consensus_thispaper1}. %A detailed comparison is provided in  \cref{subsection:append-comparison lipschitz ineq}.  
%Typically, smaller Lipschitz constant implies larger admissible stepsize $\alpha$ in RGM and a faster convergence rate. 
%For \eqref{ineq:lipschitz-local} and \eqref{ineq:lips_riemanniangrad-local}, the positive semidefinite condition  holds when all local variables $x_i$ are close to each other. For example,  since all  eigenvalues of $\half(x_i^\top x_j + x_j^\top x_i)$ are bounded above by 1, if $\Tr({x_i}^\top{x_j})\geq r-1$, or equivalently  $\normfro{x_i-x_j}^2 = 2(r-\inp{x_i}{x_j} )\leq  2$, then it is positive semi-definite.   The Lipschitz constant in \eqref{ineq:lipschitz-local} and \eqref{ineq:lips_riemanniangrad-local} is $1$, which is better than the global estimation  in \eqref{ineq:lipschitz} and \eqref{ineq:lips_riemanniangrad}. 

 By now,  we have obtained  three second-order properties   \eqref{key}, \eqref{ineq:ret_second-order} \eqref{property of projection onto tangent} in  \cref{lem:distance between Euclideanmean and IAM,lem:nonexpansive_bound_retraction,lem:relation between R and E}. These lemmas would help us to solve the non-linearity issue, and  we can get a similar Riemannian restricted secant inequality as \eqref{restricted strong convexity}. Before that, in next section we proceed to show the global convergence of \cref{alg:DRCS} with a tight estimation of the  stepsize $\alpha$.   
	
%%%%%%%%%%%%%%new section 

\section{The Global Convergence Analysis }\label{sec:consensus}

We first consider the convergence of sequence $\{\X_k\}$ generated by \cref{alg:DRCS} in this section. We build on the results of \cite{absil2005convergence,schneider2015convergence,Liu-So-Wu-2018} to provide a necessary and sufficient condition for the optimality of critical points (\cref{prop:local opt}). The main results on the local rate are presented in \cref{sec:linear rate}. 
\begin{definition}[\L{}ojasiewicz inequality]\label{def:Lojasiewicz ineq}
	We say that $\X\in\M^N$ satisfies the \L{}ojasiewicz inequality for the projected gradient $\grad f(\X)$ if there exists $\Delta>0$, $\Lambda>0$ and $\theta\in(0,1/2]$ such that for all $\Y\in\M^N$ with $\normfro{\Y-\X}<\Delta$, it holds that
	\be\label{L-inequality}
	\left|  f(\Y) - f(\X) \right|^{1-\theta} \leq \Lambda\normfro{\grad f(\X)}\tag{\L{}}. 
	\ee
\end{definition}
Since $\varphi^t(\X)$ is real analytic, and the Stiefel manifold is a compact real-analytic submanifold, it is well known that a \L{}ojasiewicz inequality holds  at each critical point of problem \eqref{opt:consensus_thispaper1}  \cite{schneider2015convergence}. Therefore, we know that the sequence  $\{\X_k\}$ converges to a single critical point with properly chosen $\alpha$. The exponent $\theta$ decides the local convergence rate.  Later we will show a similar gradient dominant inequality in \cref{prop:gradient dominant}. 
 \begin{lemma}\label{lem:sublinear-convergence}
 Let $G : = \max_{\X \in\M^N}\normfro{\grad \varphi^t (\X)}$.
	Given any  $t\geq 1$ and $\alpha\in(0,  \frac{1}{MG + L_t/2})$, where $M$ is the constant in \cref{lem:nonexpansive_bound_retraction} and $L_t$ is the Lipschitz constant in \cref{lem:lipschitz}, the sequence $\{\X_k\}$ generated by \cref{alg:DRCS} converges to a critical point of problem \eqref{opt:consensus_thispaper1} sub-linearly.  Furthermore, if some critical point is a limit point of $\{\X_k\}$  and has exponent $\theta = 1/2$ in \eqref{L-inequality}, $\{\varphi^t(\X_k)\}$ converges to $0$ Q-linearly and	the sequence $\{\X_k\}$ converges to the critical point R-linearly\footnote{A sequence $\{a_k\}$ is said to converge R-linear to $a$ if there exists a sequence $\{\varepsilon_k\}$ such that $|a_k - a| \leq \varepsilon_k$ and $\{\varepsilon_k\}$ converges Q-linearly to 0.}. 
\end{lemma}

	The proof follows \cite[Section 2.3]{schneider2015convergence} and  \cite{boumal2019global}, but here we use the descent lemma (\cref{lem:lipschitz}). It  is provided in Appendix. 

%The tangent space is given by $\T_x\M = \{\eta: \eta = x K + x^\perp Q \}$, where $Q\in\R^{r\times r}$ is any skew symmetric matrix and $Q\in \R^{(d-r)\times r}$. 
%The dimension of tangent space is $dr - \frac{1}{2}r(r+1)$. When $d-2\geq r$, we have $\dim (\T_x\M) > \dim(N_x\M)$.
\begin{remark}\label{remark: small stepsize}
The bound of stepsize $\alpha$ is $\frac{2}{2MG + L_t}$, decided by the Lipschitz constant and the constant of  retraction. It is the same as that of \cite{boumal2019global}. Compared with the result in \cite{tron2012riemannian}, the upper bound of stepsize using exponential map is only determined by $2/L_g^\prime$. In the proof, we notice that  $\alpha < 2/( 2M\normfro{\grad \varphi^t(\X_k)} + L_t )$ can guarantee the convergence. As $\lim_{k\rightarrow\infty}\normfro{\grad \varphi^t(\X_k)}=0$, finally the upper bound will be approximately $2/L_t$. 
\end{remark}

\cref{lem:sublinear-convergence} suggests the convergence to a critical point. We are more interested in the convergence to the consensus configuration. 
It is shown in \cite{markdahl2020high} that all second-order critical points of problem \eqref{opt:consensus_thispaper1} are global optima whenever $r\leq \frac{2}{3}d-1$. Therefore, the DRCS can  be guaranteed to almost always converge to the optimal point set $\mathcal{X}^*$\cite{lee2017first}. 
\begin{lemma}\cite{markdahl2020high}
	When $r\leq \frac{2}{3}d-1$, all second-order critical points of problem \eqref{opt:consensus_thispaper1}   are global optima. That is, the Riemannian Hessian at all saddle points has strictly negative eigenvalues.
\end{lemma}
The following theorem is a discrete-time version of \cite[Theorem 4]{markdahl2020high}.  It builds on \cref{lem:sublinear-convergence} and \cite[Theorem 2, Corollary 6]{lee2017first} and suggests that with random initialization, sequence $\{\X_k\}$ of  \cref{alg:DRCS} almost always converges to the consensus configuration. 
\begin{theorem}\label{thm:global convergence}
 When $r\leq \frac{2}{3}d-1$, let    $\alpha\in(0, C_{\M,\varphi^t})$, where $C_{\M, \varphi^t}: = \min\{ \frac{\hat{r}}{G} ,\frac{1}{\hat{B}}, \frac{2}{2MG+L_t}\}$, $\hat{r}$ and $\hat{B}$ are two constants related to the  retraction (defined in \cite[Prop. 9]{lee2017first}). Let $\X_0$ be a random initial point of \cref{alg:DRCS}. Then the set $\{\X_0\in\St(d,r)^N: \X_k\  \textit{converges to a point of }  \cX^*\}$ has measure $1$. 
\end{theorem}

\cref{thm:global convergence} states the almost sure convergence to consensus when $r\leq  \frac{2}{3}d-1$. 
For any $d,r$, when local agents are close enough to each other, any first-order critical point is global optimum.
\begin{proposition}\label{prop:local opt}
	Suppose that $\X$ is a first-order critical point of  problem \eqref{opt:consensus_thispaper1}.  Then, $\X$ is  a global optimal point  if and only if  there exists some $y\in \R^{d\times r}$ (with $\|y\|_2\leq 1$) such that $\inp{x_i}{y}> r-1$ for all $i=1,\ldots,N$.   Moreover, if we choose $y$   as the IAM of $\X$,  then $\X$ is  a global optimal point  if and only if 
	\[ \X\in {\cL}:=\{\X: \normfroinf{\X-\bar\X}< \sqrt{2}\}.\]
\end{proposition}
 
    When $r=1$, the region $\cL$ is the same as that of $\mathcal{S}:=\{ \X:\exists y\in\M\  \st \   \max_i d_g(x_i,y)< r^* \}$ defined in \cite{tron2012riemannian}, where $r^*:=\frac{1}{2}\min\{\mathrm{inj}\M,\frac{\pi}{\sqrt{\Delta}}\}$, $ \mathrm{inj}\M$ is the injectivity radius, and $\Delta$ is the upper bound of the sectional curvature  of $\M$. Specifically, on the sphere  $\mathrm{S}^{d-1}$, the arc length  $2r^*= \pi $ corresponds to the hemisphere, which is \textit{ the largest   convex set on $\mathrm{S}^{d-1}$}.  %If we   use the induced metric,   the geodesic distance  is given by $d_g(x,y) = \arccos\inp{x}{y}.$  Hence  the arc length  $r^*=\frac{\pi}{2}$ corresponds to the Euclidean chordal distance $\sqrt{2}$ in \cref{prop:local opt}. 
 Geometrically, it means that $x_i$ cannot be the antipode of any $x_j$, which is known as the cut locus\cite{afsari2011riemannian}. However,     $\mathrm{inj}\M$ is unknown for general case $r>1$.  In \cite{sarlette2009synchronization,tron2012riemannian,lageman2016consensus}, it was shown that the continuous Riemannian gradient flow  starting in $\cL$  converges to $\cX^*$ on sphere $\mathrm{S}^{d-1}$ and the convergence rate is linear\cite{sarlette2009synchronization,lageman2016consensus}.  
However, it is still unclear whether an algorithm could achieve global consensus initialized in $\cL$ when $r>1$. The main challenge here is that the vanilla gradient method cannot guarantee that the sequence stays in $\normfroinf{\X-\bar\X}< \sqrt{2}$. Hence, in \cite{tron2012riemannian}, there is a need to assume $\X_0\in \mathcal{S}_{\text{conv}}:=\{ \phi(\X)< \frac{(r^*)^2}{2dia(\G)} \}$, where $\phi(\X)$ is the objective in \eqref{opt:consensus_prev} with $\dist$ being the geodesic distance and $ dia(\G)$ is the diameter of the graph $\G$. But $\mathcal{S}_{\text{conv}}$ is smaller than $\cS$. Here, we present the same result on $\mathrm{S}^{d-1}$ with a different proof since we work with Euclidean distance. We cannot generalize the proof to $r>1$.

\begin{lemma} \label{lem:stay in L sphere}
	Let $r=1$ and assume that there exists a $y\in\St(d,1)$ such that the initial point $\X_0$ satisfies $ \inp{x_{i,0}}{y} \geq  \delta, \quad  \forall i\in[N]$
	for some $\delta >0 $. 
	Then, the sequence $\{\X_k\}$ generated by \cref{alg:DRCS} with $\alpha \leq 1$ and  $t\geq 1$  satisfies 
	\be\label{ineq:lower bound inner product} \inp{x_{i,k}}{y} \geq  \delta, \quad  \forall i\in[N], \ \forall k\geq 0.\ee
\end{lemma}

%
%\begin{lemma}\label{lem:help lemma}
%	For $x,y\in\St(d,r)$, if $\normfro{x-y}^2\leq 2c_0$, where $c_0\in(0,1/2)$, we have
%	$ \half (x^\top y + y^\top x)  \succcurlyeq (1- c_0)I_r. $
%\end{lemma}
%\begin{proof}[\textbf{Proof of \cref{lem:help lemma}}]
%	Since $x,y\in\St(d,r)$, all eigenvalues of $ \half (x^\top y + y^\top x) $ are upper bounded by $1$. Noting that $\normfro{x-y}^2 = 2r-\Tr(x^\top y + y^\top x) \leq 2c_0,$
%	we get $ \half \Tr (x^\top y + y^\top x) \geq r-  c_0. $
%	This suggests that
%	$ \half   (x^\top y + y^\top x) \succcurlyeq (1-  c_0)I_r. $
%\end{proof}
\section{Local Linear Convergence}\label{sec:linear rate}
As we see in \cref{prop:local opt}, the region $\cL$ characterizes the local landscape of \eqref{opt:consensus_thispaper1}.  
 Typically,  a local linear rate can be obtained for RGM  if the Riemannian Hessian is non-singular at global optimal points. The Riemannian Hessian of $\varphi^t(\X)$ is a linear operator. For any tangent vector $\eta^\top = [ \eta_1^\top,\ldots,\eta_n^\top]$, we have  \cite{absil2013extrinsic}
\be\label{Riemannian hessian-0}
\bad
\quad &\inp{\eta}{ \Hess \varphi^t(\X)[\eta]}  
%= &  \normfro{\eta}^2-\sum_{i=1}^N \sum_{j=1}^N    W^t_{ij}\inp{\eta_i}{ \eta_j}   - \sum_{i=1}^N \inp{\eta_i}{\eta_i x_i^\top \p_{N_{x_i}} \nabla \varphi^t(x_i)}\\
= \normfro{\eta}^2-\sum_{i=1}^N \sum_{j=1}^N    W^t_{ij}\inp{\eta_i}{ \eta_j}   - \sum_{i=1}^N \langle \eta_i,\eta_i  (\frac{1}{2}[ \nabla \varphi^t_i(\X)^\top x_i + x_i^\top \nabla \varphi^t_i(\X) ])\rangle.  
\ead
\ee
%See \eqref{Riemannian hessian-1} in the appendix for details of derivations.  
Following \cite{markdahl2020high}, if we let   $x_1=\ldots=x_N$ and $\eta_i = \p_{\T_{x_i}\M} \xi$ for any $\xi\in\R^{d\times r}$,  \eqref{Riemannian hessian-0} reads $0= \sum_{i=1}^N \inp{\eta_i}{ \Hess \varphi^t_i(\X)[\eta_i]   }$. Therefore, similar as the Euclidean case, the Riemannian Hessian at any consensus point has a zero eigenvalue. This motivates us to consider an alternative to the strong convexity. Luckily, there are more relaxed conditions (than strong convexity) for Euclidean problems. % In this section, we derive a restricted secant inequality (RSI) for problem \eqref{opt:consensus_thispaper1}, which is a generalization of the RSI discussed in \cite{zhang2013gradient}, which takes the following form
%{ \blue{ More specifically, the RSI in Euclidean space takes the form of 
%\be%\label{def:Euclidean RSI}
%\inp{\X- \p_{\cX^*}(\X)}{\nabla f(\X) - \nabla f(\p_{\cX^*}(\X))}\geq c_0 \normfro{\X-\X^*}^2, \nonumber
%\ee 
%where $\p_{\cX^*}(\X)$ is the projection of $\X$ onto the optimal set $\cX^*$ of function $f(\X)$ and $c_0 > 0$ is some constant. 
%Since $\nabla f(\p_{\cX^*}(\X)) = \mathbf{0}$, it can be written as
%\[
%\inp{\X- \p_{\cX^*}(\X)}{\nabla f(\X) }\geq c_0 \normfro{\X-\X^*}^2.
% \]

To exploit this, in the next subsection, we will generalize the inequality \eqref{restricted strong convexity} to its Riemannian version as follows 
\be\label{def:Riemannian RSI-0}
\bad
&\quad \inp{\X- \p_{\T_\X\M^N}\bar\X}{\grad \varphi^t(\X)} \geq c_d \normfro{\X-\bar\X}^2 + c_g \normfro{\grad\varphi^t (\X)}^2,
\ead\ee 
where $c_d>0, c_g > 0$ and $\X$ is in some neighborhood of $\cX^*$. Note that for the Riemannian problem \eqref{opt:consensus_thispaper1}, we need to substitute the Euclidean gradient with Riemannian gradient. Moreover,  the IAM $\bar\X $ should be mapped into the tangent space $\T_{\X}\M^N$. One can use the inverse of exponential map  $\Exp^{-1}_{\X}(\bar\X)$.  However, the  map  $\Exp^{-1}_{\X}(\bar\X)$ is difficult to compute. Note that $\Exp_{\X}$ is a local diffeomorphism. By the inverse function theorem,  we have  $\Exp^{-1}_{\X}(\bar\X) = \bar \X - \X + \mathcal{O}(\normfro{\X-\bar\X}^2) $.   
 Using the property in \eqref{property of projection onto tangent}, we know that $\p_{\T_{\X}\M^N}(\bar \X - \X)$ is a second-order  approximation to $\Exp^{-1}_{\X}(\bar\X)$. As such, we directly project  $\bar\X$ onto the tangent space of $\X$. Note that this is not the inverse of any retraction. 
Moreover, since
\[
\bad
 \inp{\X- \p_{\T_\X\M^N}\bar\X}{\grad \varphi^t (\X)}  
%& = \inp{\p_{\T_\X\M^N} (\X- \bar\X)}{\p_{\T_\X\M^N}\nabla f(\X)}\\
   = \inp{\X-  \bar\X}{\grad \varphi^t (\X)},
\ead
\] 
we will investigate the following formal definition of RSI 
 \be\label{def:Riemannian RSI}
 \inp{\X-  \bar\X}{\grad \varphi^t (\X)}\geq  c_d \normfro{\X-\bar\X}^2 + c_g \normfro{\grad\varphi^t (\X)}^2. \tag{RSI}
 \ee 
To establish the \eqref{def:Riemannian RSI}, we first show  the quadratic growth (QG) property of $\varphi^t(\X)$ (\cref{prop:quadratic growth}).	In the Euclidean space, especially for convex problems, QG condition is equivalent to the RSI  as well as the \L{}ojasiewicz inequality with $\theta=1/2$\cite{karimi2016linear}. To the best of our knowledge, QG cannot be used directly to establish the linear rate of GD and it is usually required to show the equivalence to Luo-Tseng\cite{luo1993error} error bound inequality (ERB)\cite{drusvyatskiy2018error}.  However, for nonconvex problems, RSI is usually stronger than QG. We will discuss more about this later. 
 	\begin{lemma}[Quadratic growth]\label{prop:quadratic growth}
% 	For any $t\geq 1$, consider a neighborhood $\N_{Q,t}$ of $\cX^*$, defined by
% 	\be\label{def:region quadratic growth}
% 	\N_{Q,t}: = \{ \X: \normfro{\X-\bar\X}^2\leq  {N} \delta_t^2\}, \quad  0\leq \delta_t\leq \frac{\sqrt{1-\sigma_2^t}}{2\sqrt{r}}. 
% 	\ee  For $\X\in\N_{Q,t}$, the following inequality holds
	For any $t\geq 1$ and $\X\in\M^N$, we have 
 	\be\label{ineq:quadratic growth_global}
 	\varphi^t(\X) - \varphi^t(\bar\X)   
 	\geq  \frac{\mu_t}{2}\normfro{\X-\hat\X}^2\geq \frac{\mu_t}{4}\normfro{\X-\bar\X}^2,\tag{QG}
 	\ee
 		where the constant is given by \[ \mu_t := 1-\lambda_{2}(W^t). \] The $\lambda_{2}(W^t)$ is the second largest eigenvalue of $W^t$, and $L_t$ is given in \cref{lem:lipschitz}.
Moreover, if $\normfro{\X - \bar\X}^2\leq \frac{N}{8r}$, we have
\be\label{ineq:quadratic growth_local}
\varphi^t(\X) - \varphi^t(\bar\X)   
\geq     \frac{\mu_t}{2}(1-\frac{4r}{N}\normfro{\X-\bar{\X}} ^2)\normfro{\X-\bar{\X}} ^2.\tag{QG'}
\ee
 \end{lemma}

The second inequality \eqref{ineq:quadratic growth_local} is a local quadratic growth property, which is  tighter than \eqref{ineq:quadratic growth_global}.

\subsection{Restricted Secant Inequality}
In this section, we discuss how to establish \eqref{def:Riemannian RSI}.  We will derive RSI in the following forms
\be\label{def:Riemannian RSI-1}
\inp{\X-  \bar\X}{\grad \varphi^t (\X)}\geq  c_d^\prime  \normfro{\X-\bar\X}^2, \quad c_d^\prime >0 \tag{RSI-1}
\ee 
and
\be\label{def:Riemannian RSI-2}
\inp{\X- \bar\X}{\grad \varphi^t (\X)}\geq  c_g^\prime \normfro{\grad\varphi^t (\X)}^2  \quad c_g^\prime >0.
\tag{RSI-2}
\ee 
Then, \eqref{def:Riemannian RSI} can be obtained by any convex combination of \eqref{def:Riemannian RSI-1} and \eqref{def:Riemannian RSI-2}.   
To proceed the analysis, we  define for $i\in[N]$
\begin{align}\label{eq:pi}
   p_i := \frac{1}{2} ( x_{i} -\bar x)^\top (x_i - \bar x), 
\end{align}
and
\begin{align}\label{eq:qi}
   q_i :=   \half \sum_{j=1}^N   W_{ij} ^t  (x_i-x_j)^\top (x_i-x_j). 
\end{align}
Let  $\Y=\bar\X$ in \eqref{ineq:key relation between R and E}. 
We get
\be\label{ineq:RSI first look}
\bad
\quad \inp{\grad \varphi^t(\X)} {\X-  \bar\X} &=  \inp{\nabla \varphi^t(\X)}{ \X - \bar \X}   -   \sum_{i=1}^N\inp{p_i }{q_i } \\
& =   2\varphi^t(\X)  -   \sum_{i=1}^N\inp{p_i }{q_i },
\ead \ee
where in the last equation we used the following two  identities 
$2\varphi^t(\X)=\inp{\nabla \varphi^t(\X)}{ \X }\footnote{See \eqref{eq:function and gradient} in Appendix.}$ 
and $ \inp{\nabla \varphi^t(\X)}{ \bar\X  } = 0. $
 The term $\sum_{i=1}^N\inp{p_i }{q_i } $ is non-negative, so if we substitute \eqref{ineq:RSI first look} into \eqref{def:Riemannian RSI}, we observe that RSI is stronger than QG. Moreover, by Cauchy-Schwarz inequality, we have
%\begin{enumerate}
% \item \be\label{ineq:bound type 1} \bad
% \sum_{i=1}^N\inp{p_i }{q_i } &\leq \sum_{i=1}^N \normfro{p_i}   \cdot  \max_{i\in[N]} \normfro{q_i} \\
% &\leq  \normfro{\X-\bar \X}^2 \cdot  \max_{i,j} \normfro{x_i - x_j}^2. 
% \ead\ee
% \item 
\be\label{ineq:bound type 1}
 \bad
  \sum_{i=1}^N\inp{p_i }{q_i } \leq    \max_{i\in[N]} \normfro{p_i} \cdot 2\varphi^t(\X)\leq \varphi^t(\X)\cdot\normfroinf{\X-\bar \X}^2. 
 \ead\ee
% or may be like 
% \be\label{ineq:bound type 3}
% \bad
% \sum_{i=1}^N\inp{p_i }{q_i } \leq    \max_{i\in[N]} \normtwo{p_i} \cdot 2\varphi^t(\X)\leq \varphi^t(\X)\cdot\normtwoinf{\X-\bar \X}^2. 
% \ead\ee
%and for $r=1$ specifically, {\rd the following holds}
%\be\label{ineq:high order r=1}
%\bad
%\sum_{i=1}^N\inp{p_i }{q_i } =   \sum_{i=1}^N \varphi_i(\X) \cdot \normfro{x_i-\bar x}^2. 
%\ead\ee
Hence, we see that if $\normfroinf{\X - \bar \X}< \sqrt{2}$, we have $\inp{\grad \varphi^t(\X)} {\X-  \bar\X}> 0$, which implies that the direction $-\grad \varphi^t(\X)$ is positively correlated with the direction $\bar\X - \X$. However, it seems difficult to  guarantee  $\normfroinf{\X_k - \bar \X_k}< \sqrt{2}$ since $\bar\X_k$ is not fixed. We will see in \cref{lem:stayinbox} that multi-step consensus can help us circumvent this problem. Moreover, note that 
 \be\label{ineq:bound type 2}
 \bad
 \sum_{i=1}^N\inp{p_i }{q_i } \leq \varphi^t(\X)\cdot\normfro{\X-\bar \X}^2, 
 \ead\ee
 so we can also establish \eqref{def:Riemannian RSI-1}
 when $\varphi^t(\X)=O(\mu_t)$, as we will see in \cref{lem:regul_cond}. 

To conclude, the two inequalities \eqref{ineq:bound type 1} and \eqref{ineq:bound type 2} correspond to  two  neighborhoods of $\cX^*$: $\N_{R,t}$ and $N_{l,t}$, which are   defined  in the sequel. First,  we define
\be\label{contraction_region}    \N_{R,t}: = \N_{1,t}\cap \N_{2,t}, \ee
where 
\begin{align}
	\N_{1,t} :& = \{ \X: \normfro{\X-\bar\X}^2 \leq N\delta_{1,t}^2  \} \label{contraction_region_1} \\
	 \N_{2,t} : &= \{ \X: \normfroinf{\X-\bar\X} \leq \delta_{2,t}  \} \label{contraction_region_2}
\end{align}
%	  \[ \normfroinf{(W\otimes I_N) (\X-\bar\X)}\leq \delta_3\sqrt{r}  \]
    and $\delta_{1,t},\delta_{2,t}$
satisfy
\be\label{delta_1_and_delta_2} \delta_{1,t} \leq \frac{1 }{ 5\sqrt{r}}\delta_{2,t} \quad \textit{and} \quad    \delta_{2,t} \leq  \frac{1}{6}.  \ee
Secondly, the region $\N_{l,t}$ is given by 
  \be\label{def:N l,t} \N_{l,t}:=\{ \X:   \varphi^t(\X)\leq  \frac{\mu_t }{4}\}\cap \{ \X: \normfro{\X-\bar\X}^2 \leq N\delta_{3,t}^2  \},\ee where     $\delta_{3,t}$ 	satisfies
\be\label{delta_3} \delta_{3,t} \leq \min \{ \frac{1}{\sqrt{N}},\frac{1}{4 \sqrt{r}} \}.  \ee
According to \cref{prop:local opt}, the radius  $\N_{2,t}$ cannot be larger than $\sqrt{2}$, which  is  the manifold property,  while $\N_{l,t}$ is decided by the connectivity of the network. If the connectivity is stronger, then the region is  larger. The \eqref{def:Riemannian RSI-1} is formally established in the following lemma.

\begin{lemma}\label{lem:regul_cond}
	Let $\mu_t$ be the constant given in \cref{prop:quadratic growth} and $t\geq 1$. 
	\begin{enumerate}
		\item 	Suppose $\X\in\N_{R,t}$, where $\N_{R,t}$ is defined by \cref{contraction_region}.  There exists   a constant $\gamma_{R,t}>0$:
		$$ \gamma_{R,t} :=   (1- 4r\delta_{1,t}^2) (1- \frac{\delta_{2,t}^2}{2}) \mu_t\geq \frac{\mu_t}{2} ,$$
		such that the following holds:
		\begin{equation}\label{ineq:reg_condition}
		\bad
		 \inp{ \X- \bar{\X}  }{\grad \varphi^t(\X)}
		\geq 
		  \gamma_{R,t} \normfro{\bar{\X}-\X}^2.
		\ead
		\end{equation}
	
		\item  For $\X\in\N_{l,t}$, where $\N_{l,t}$ is defined by \eqref{def:N l,t},
				we also have \eqref{def:Riemannian RSI-1}, in which 
				$c'_d=\gamma_{l,t} :=  \mu_t(1- 4r\delta_{3,t}^2) -   \varphi^t(\X)   \geq \frac{\mu_t}{2} . $
	\end{enumerate}

\end{lemma}
\begin{remark}
%	 The main idea of the proof of \cref{lem:regul_cond} is to use $\grad \varphi^t(\X) = \nabla \varphi^t(\X) - \p_{N_{x_i}\M} \nabla \varphi^t(\X)$ and to show $\inp{\nabla \varphi^t(\X)}{\X-\bar\X} = (1-\sigma_2^t)\cdot\mathcal{O}(\normfro{\X-\bar\X}^2)$, $\inp{\p_{N_{x_i}\M}\nabla \varphi^t(\X)}{\X-\bar\X} = (1-\sigma_2^t)\cdot\mathcal{O}( \normfro{\X-\bar\X}^2)$ in the local region. The proof is provided in the Appendix. ]
We show $\gamma_{R,t}$ and $\gamma_{l,t}$ by combining \eqref{ineq:quadratic growth_local} with \eqref{ineq:bound type 1} and \eqref{ineq:bound type 2}, repectively. 
	For \eqref{ineq:reg_condition}, $ (1- 4r\delta_{1,t}^2) (1- \frac{\delta_{2,t}^2}{2}) \geq\frac{ 1 }{2}$  suffices to guarantee the lower bound. However, we impose \eqref{delta_1_and_delta_2} to guarantee $\X_k\in\N_{2,t}$ for all $k\geq 0$. Moreover, we find that by combining \eqref{ineq:quadratic growth_global} with \eqref{ineq:bound type 1}, one can also get \eqref{def:Riemannian RSI-1} without the constraint  $\N_{1,t}$. But the coefficient will be smaller. For simplity, we only show the results that stay in $\N_{1,t}\cap\N_{2,t}$.  Similarly for $\N_{l,t}$, $ \delta_{3,t} \leq\frac{1}{4 \sqrt{r}}$ is enough to ensure RSI. We impose $\delta_{3,t}\leq 1/\sqrt{N}$ to get \cref{lem:helpful lem to prove X in N_Qt} which is useful to ensure $\X_k\in\N_{l,t}$.  In fact,  $\delta_{3,t}\leq 1/\sqrt{N}$ does not shrink the region since $\varphi^t(\X)\leq \mu_t$ implies a small region by \cref{prop:quadratic growth}. Also, since $\delta_{3,t}\leq 1/\sqrt{N}$, it is clear that $\N_{l,t}$ is  smaller than $\N_{R,t}$ when $N$ is large enough.
\end{remark}

%  In \cite{tron2012riemannian},  a local   region $\cS_{\text{conv}}:=\{ \phi(\X)< \frac{(r^*)^2}{2dia(G)} \}$  was studied, where $\phi(\X)$ is defined by \eqref{opt:consensus_prev} and $ dia(G)$ is the diameter of the   $G$.  If the initial point $\X_0$ satisfying $\X_0\in \mathcal{S}_{\text{conv}},$ then RGD  converges sublinearly to global optimal point.  \\   % Let us compare the region of 1) and 2) with that of \cite{tron2012riemannian}. 

 \cref{lem:regul_cond} also  implies that the following error bound inequality holds for $\X\in\N_{R,t}$ and  $\X\in\N_{l,t}$ 
\be\label{ineq:error bound} \normfro{\X-\bar{\X}} \leq \frac{2}{\mu_t}\normfro{\grad \varphi^t(\X)}. \tag{ERB} \ee
This inequality is a generalization of the Luo-Tseng  error bound\cite{luo1993error}  for problems in Euclidean space.  In \cite{karimi2016linear}, the following holds for smooth non-convex problems
\[ \text{RSI}  \Rightarrow  \text{ERB}  \Leftrightarrow  \text{\L{}ojasiewicz inequality with }  \theta = 1/2     \Rightarrow  \text{QG}.\] However, in Euclidean space and for convex problems, they are all equivalent. RSI can be used to show the Q-linear rate of $\dist(\X,\cX^*)$, and ERB can be used to establish the Q-linear rate of the objective value and the R-linear rate of $\dist(\X,\cX^*)$.    Moreover, under mild assumptions  QG and ERB  are shown to be equivalent for second-order critical points for Euclidean nonconvex problems \cite{yue2019quadratic}. 
%It tells us that any first-order critical point in $\N_{R,t}$ is also a global optimal point. 
%Compared with \cref{prop:second_is_global}, the region $\N_{R,t}$ is smaller, but the inequality \eqref{ineq:reg_condition} can be used to establish the local linear rate  of RGD for solving the consensus problem. 
Some other error bound inequalities are also obtained over the Stiefel manifold or oblique manifold. For example, Liu\etal \cite{Liu-So-Wu-2018} established the error bound inequality  of any first-order critical point for the eigenvector problem. And   \cite{Liu-generalized-power-phase-synchronization-2017,chen2019first}   gave  two types of error bound inequality for phase synchronization problem. Our proof of \cref{lem:regul_cond} relies mainly on the doubly stochasticity of $W^t$ and the properties of IAM, thus it is fundamentally different from previous works. Another similar form of RSI is the Riemannian  regularity condition   proposed in \cite{zhu2019linearly} for minimizing the nonsmooth problems over Stiefel manifold. %The nonsmooth version of  RRC  is given as $\inp{\grad f(x)}{x-x^*} \geq c\normfro{x-x^*}, c>0$, where $\grad f(x)$ is any Riemannian subgradient and $x^*$ is a global minimizer.  Since the objective \eqref{opt:consensus_thispaper1} is smooth, the right handside of \eqref{def:Riemannian RSI} appears as a quadratic term.  

 Following the same argument as \cite{Liu-So-Wu-2018}, the error bound inequality \eqref{ineq:error bound} implies a growth  inequality similar as   \L{}ojasiewicz  inequality. However, the neighborhoods $\N_{R,t}$ and $\N_{l,t}$ are relative to the set $\cX^*$, which is different from the \cref{def:Lojasiewicz ineq}.   It can be used to show the Q-linear rate of $\{\varphi^t(\X_k)\}$ only if $\X_k\in\N_{R,t}$ or $\X_k\in\N_{l,t}$ can be guaranteed. 
 \begin{proposition}\label{prop:gradient dominant}
 	For any $\X\in\N_{R,t}$ or $\X\in\N_{l,t}$ it holds that
 	\be\label{ineq:grad dominant}   \varphi^t(\X)  \leq \frac{3}{2\mu_t}\normfro{\grad  \varphi^t(\X)}^2. \ee
 \end{proposition}

We need the following bounds for $\grad \varphi^t(\X)$ 
 by noting that $\varphi^t(\X)$ is Lipschitz smooth as shown in \cref{lem:lipschitz}. It will be helpful to show \eqref{def:Riemannian RSI-2}.  

\begin{lemma}\label{lem:bound_of_grad}
	For any $\X\in\St(d,r)^N$, it follows that
	\be\label{ineq:bound-of-sum-gradh}
	\begin{aligned}
		\normfro{ \sum_{i=1}^{N}\grad \varphi^t(x_{i})}
		\leq L_t\normfro{\X -\bar \X}^2 
	\end{aligned} \ee
	and
	\be\label{ineq:bound-of-gradh} \normfro{\grad \varphi^t(\X)}^2\leq 2L_t\cdot\varphi^t(\X),\ee
	where $L_t$ is the Lipschitz constant given in \cref{lem:lipschitz}. Moreover, suppose $\X\in\N_{2,t}$, where $\N_{2,t}$ is defined by \eqref{delta_1_and_delta_2}. We then have
	\be\label{ineq:bound-of-gradh-i}  \max_{i\in[N]}\normfro{\grad \varphi^t_i(\X)}\leq 2 \delta_{2,t} . \ee	
\end{lemma}

Next, we  are going to show \eqref{def:Riemannian RSI-2}. The two RSI's are crucial to show that $\X_k\in\N_{1,t}$ or $\X_k\in\N_{l,t}$ with stepsize $\alpha = \mathcal{O}(\frac{1}{L_t})$. This holds naturally for convex problems in Euclidean space\cite{zhang2013gradient}, but it holds only locally for problem \eqref{opt:consensus_thispaper1}. 
\begin{proposition}[Restricted secant inequality]\label{lem:helpful lem to prove X in N_Qt}
	The following two inequalities hold for $\X\in \N_{R,t}$ and $\X\in\N_{l,t}$
	\begin{equation}\label{ineq: lower bound by gradient}
	\inp{\X-\bar\X}{\grad \varphi^t(\X)}\geq \frac{\Phi}{2L_t}\normfro{\grad \varphi^t(\X)}^2,
	\end{equation}
	and 
	\begin{equation}\label{ineq: lower bound by gradient and distance}\tag{RSI-I}
	\bad
	\inp{\X-\bar\X}{\grad \varphi^t(\X)}\geq \nu\cdot\frac{\Phi}{2L_t}\normfro{\grad \varphi^t(\X)}^2 + (1-\nu) \gamma_{t}\normfro{\X-\bar\X}^2,
	\ead\end{equation}
	for any $\nu\in [0,1]$, where $\gamma_t$ and $\Phi>1$ are   constants related to $\X$, which are given by\be\label{def:constant phi}
	\gamma_t :=	\left\{\begin{matrix}
		\gamma_{R,t},	& \X\in\N_{R,t} \\ 
		\gamma_{l,t},	&  \X\in\N_{l,t},
	\end{matrix}\right.
	\ee	
	\be\label{def:constant phi}
	\Phi :=	\left\{\begin{matrix}
		2 - \normfroinf{\X-\bar\X}^2,	& \X\in\N_{R,t} \\ 
		2 - \normfro{\X-\bar\X}^2,	&  \X\in\N_{l,t}.
	\end{matrix}\right.
	\ee	
\end{proposition}

\subsection{ Local Rate of  Consensus} \label{subsec:rate of consensus}
%In this section, we use the QG property in \cref{prop:quadratic growth} and RSI \cref{lem:regul_cond} to show the local linear rate. Given QG property, the Q-linear rate of $\{\varphi^t(\X_k)\}$ and R-linear rate of $\{ \normfro{\X_k-\bar\X_k}^2\}$ can be proved easily, only if we can show $\X_k\in\N_{Q,t}$. This can be guaranteed by establishing standard inequalities hold for Euclidean problem.  
Endowed with the RSI condition, we can now solve the problem \eqref{opt:consensus_thispaper1}. The main difficulty now is to show that $\X_k \in \N_{2,t}$. In the literature, there have been some work discussing how to bound the infinity norm for Euclidean gradient descent (e.g., \cite{ma2019implicit,zhong2018near}), which is  called the implicit regularization \cite{ma2019implicit}. This is often related to a certain incoherence condition under specific statistical models. However, to solve \eqref{opt:consensus_thispaper1}, we use the Riemannian gradient method and we need to verify this property for DRCS. %{\rd When $r=1$, based on \cref{lem:stay in L sphere} we can show $\X_{k+1}\in \cL$ if $\X_k\in \cL$ with $\alpha \leq 1 $ and any $t\geq 1$. However, the region $\N_{2,t}$ is defined by $\bar x$ which is not fixed. We will see that $\normtwo{W^t -  \frac{1}{N} \mathbf{1}_N \mathbf{1}_N^\top} = \sigma_2^t = \mathcal{O}( \frac{1}{\sqrt{N}})$ is needed to guarantee the incoherence. That is why we need to employee multi-step consensus. }
We have the following bound in \eqref{lem:bound_of_2_infty}  for the  total variation distance between any row of $W^t$ and the uniform distribution.

\begin{lemma}\label{lem:total deviation from 1/N}
	Given any $\X\in\N_{2,t}$, where $\N_{2,t}$ is defined in \eqref{contraction_region_2}, 
	%		   \[ \max_{i\in[N]} \normfro{\sum_{j=1}^N (W^t_{ij}-1/N)x_j}\leq  \sigma_2^t \delta_2\sqrt{Nr} .  \]
	%		   Hence, 
	if $t\geq \lceil \log_{\sigma_2}(\frac{1}{2\sqrt{N}})\rceil$, we have
	\be\label{lem:bound_of_2_infty} \max_{i\in[N]} \normfro{\sum_{j=1}^N (W^t_{ij}-1/N)x_j}\leq \frac{\delta_{2,t}}{2}.  \ee
\end{lemma}
The lower bound $\lceil\log_{\sigma_2 }(\frac{1}{2\sqrt{N}})\rceil$ may not be a small number. For example, when $W$ is the lazy Metropolis matrix of regular   connected graph, $\sigma_2$ usually scales as $1-\mathcal{O}(\frac{1}{N^2})$ \cite[Remark 2]{nedic2017achieving} and $\log_{\sigma_2 }(\frac{1}{2\sqrt{N}}) =\mathcal{O}(N^2\log N)$. However, for example, for a star graph this can be $\mathcal{O}(\log N)$. It will be interesting to see under what conditions \eqref{lem:bound_of_2_infty} holds for $t=1$ as a future work. Here, we require this condition to ensure the algorithm is in a proper local neighborhood.

Following a perturbation lemma of the polar decomposition \cite[Theorem 2.4]{li2002perturbation}, we get the following technical lemma which will be useful to bound the Euclidean distance between two consecutive points $\bar{x}_k$ and $\bar{x}_{k+1}$. 

\begin{lemma} \label{lem:bound_of_two_average}
	Suppose $\X, \Y\in \N_{1,t}$,
	%If $ \normfroinf{\X- \Y}< 1-2\delta_{1,t}^2 $	where $\delta_{1,t}$ is given in \eqref{delta_1_and_delta_2},   
	we have	\[\normfro{\bar{x} - \bar{y}}\leq \frac{1}{1-2\delta_{1,t}^2} \normfro{\hat x - \hat y},  \]
	where $\bar{x}$ and $\bar{y}$ are the IAM of $x_1,\ldots,x_N$ and $y_1,\ldots,y_N$, respectively. 
\end{lemma}

Now, we  are ready to  prove that $\X_k$ always stays in $\N_{R,t}=\N_{1,t}\cap\N_{2,t}$ if the stepsize $\alpha$ satisfies  $0\leq \alpha\leq\min\{ \frac{1}{L_t},1, \frac{1}{M}\}$ and  $t\geq \lceil\log_{\sigma_2 }(\frac{1}{2\sqrt{N}})\rceil$. The upper bound $\frac{1}{M}$ and $1$ come from showing $\X_k \in \N_{2,t}$.
\begin{lemma}[Stay in $\N_{R,t}$]\label{lem:stayinbox}
	Let $\X_k\in\N_{R,t}$, $0\leq \alpha\leq\min\{  \frac{\Phi}{L_t} , 1, \frac{1}{M}\}$ and $t\geq \lceil\log_{\sigma_2 }(\frac{1}{2\sqrt{N}})\rceil$, where the radius of $\N_{R,t}$ is given by \eqref{delta_1_and_delta_2} and  $M$ is given in \cref{lem:nonexpansive_bound_retraction}. We then have  $\X_{k+1}\in\N_{R,t}.$ 
\end{lemma}
From the above result, we see that the stepsize is upper bounded by $\frac{\Phi}{L_t}$ and $\frac{1}{M}$, and they reflect  the role of the network and the manifold. The condition $\alpha\leq 1/L_t$ guarantees that $\X_k\in\N_{1,t}$ and $\alpha \leq \min\{1,1/M\}$ ensures that $\X_k\in\N_{2,t}.$ As we mentioned in \cref{remark:boundness on retraction}, we have $M=1$ in \eqref{ineq:ret_second-order} for the polar retraction if $ \alpha\normfro{\grad \varphi^t(x_{i,k})}\leq 1$. By our choice of $\alpha\leq 1$ and $\X_k\in\N_{R,t}$, we indeed have $ \alpha \normfro{\grad \varphi^t(x_{i,k})}\leq 2\delta_{2,t}\leq 1$ according to \cref{lem:bound_of_grad}. However, we  do not plan to remove the term $\frac{1}{M}$. Note that if we use other retractions, the bound will be slightly worse due to larger $M$ and the extra second-order term in \eqref{ineq:ret_nonexpansive-1}. 
Now, we are ready to establish  the local Q-linear convergence rate of \cref{alg:DRCS}.
\begin{theorem}\label{thm:linear_rate_consensus}
  Under \cref{assump:doubly-stochastic}. 
 (1). Let $\nu\in(0,1)$ and the stepsize  $\alpha$ satisfy $0<\alpha\leq\min\{ \frac{\nu\Phi}{ L_t}, 1,\frac{1}{M}\}$ and $t\geq \lceil\log_{\sigma_2}(\frac{1}{2\sqrt{N}})\rceil$.
	The sequence $\{\X_k\}$ of \cref{alg:DRCS} achieves consensus linearly if the initialization satisfies 
	$\X_0\in \N_{R,t}$ defined by \eqref{delta_1_and_delta_2}. That is, we have  $\X_k\in\N_{R,t}$ for all $k\geq 0$ and 
	\begin{align}\label{linear rate}
	\normfro{\X_k - \bar{\X}_k}^2 \leq (1-    2\alpha(1-\nu)  \gamma_{t})^k\normfro{\X_0 - \bar{\X}_0}^2.
	\end{align}
	Moreover, if $\alpha\leq\frac{1}{2MG+L_t}$, $\bar\X_k$ also converges to a single point.  \\
(2). 	If $\X_0\in\N_{l,t}$ and $\alpha\leq \min \{  \frac{2}{L_t+2MG},\frac{\Phi}{2L_t}\}$, one has \eqref{linear rate} for any $t\geq 1$.

\end{theorem}
  
 Combining \cref{thm:linear_rate_consensus} with \cref{lem:sublinear-convergence} and \cref{thm:global convergence}, we conclude the following results. When $\alpha<\min\{ C_{\M,\varphi^t}, \frac{2}{L_t+2MG}, \frac{\nu\Phi}{ L_t}, 1\} $ and $r\leq \frac{2}{3}d-1$, we know that with random initialization, $\{\X_k\}$ firstly converges sub-linearly and then linearly for any $t\geq 1$. We find that any $\alpha \in(0,2/L_t) $ can guarantee the global convergence in practice.  This could be explained as follows. When $\normfro{\X_k - \bar\X_k}\rightarrow 0$, then $\varphi^t(\X_k)\rightarrow 0$.  We then have $\normfro{\grad \varphi^t(\X_k)}\rightarrow 0$ (by \eqref{ineq:bound-of-gradh}), $\max_{i\in[N]}\normfro{\grad \varphi^t_i(\X)}\rightarrow 0$ and $\Phi\rightarrow 2$. Combined with \cref{remark: small stepsize} and the discussion after \cref{lem:stayinbox}, we deduce that the upper bound of $\alpha$ is asymptotically $\min\{ C_{\M,\varphi^t},   \frac{2\nu}{ L_t},1 \} $. Finally, we  also have $  \gamma_{t}\rightarrow \mu_t$ for both cases of \cref{lem:regul_cond}. If we let $\nu=1/2$ and $\alpha = 1$ is available then it implies a linear rate of $(1-\mu_t)^{1/2}$, but this could be worse than the rate $\sigma_2^{ t}$ of Euclidean consensus. We will discuss in the next section how to obtain this rate.  
 
 \subsection{Asymptotic Rate}
  To  get  the   rate of $\sigma_2^{t}$, we need to ensure  $c_d = \frac{\mu_t L_t}{\mu_t + L_t}$ and  $c_g = \frac{1}{\mu_t + L_t}$ in \eqref{def:Riemannian RSI}.
 	We will combine \cref{prop:quadratic growth} with \cref{lem:relation between R and E} to show this asymptotically for any $\X \in \N_{l,t}$.  Firstly, by \eqref{ineq:RSI first look} we have
 \be\label{ineq:RSI second look}\bad
  \inp{\grad \varphi^t(\X)} {\X-  \bar\X} 
 =  \inp{\nabla \varphi^t(\X)}{ \X - \hat \X}  
 -   \sum_{i=1}^N\inp{p_i }{q_i },
 \ead \ee
 where $p_i$ and $q_i$ are given in \eqref{eq:pi}-\eqref{eq:qi}.
 %where we use $\sum_{i=1}^N \varphi_i^t(\X) =0$ and 
 %\[  \inp{\nabla \varphi^t(\X)}{  \bar \X} =   \langle\sum_{i=1}^N \varphi_i^t(\X) ,  \bar x\rangle = 0 =  \langle\nabla \varphi^t(\X),  \hat \X\rangle.  \]
  Using \eqref{restricted strong convexity} and \eqref{ineq:relation_eum_manifmean_1} yields
 \be\label{Rieman restricted strong convexity-1}
 \bad
 & \quad \inp{\nabla \varphi^t (\X)}{\X - \hat\X}\\
 &\geq \frac{\mu_t L_t}{\mu_t + L_t}\normfro{\X - \hat\X}^2 + \frac{1}{\mu_t + L_t} \normfro{\nabla \varphi^t(\X)}^2\\
 &  \geq  \frac{\mu_t L_t}{\mu_t + L_t}(1-\frac{4r}{N}\normfro{\X - \bar\X}^2)\normfro{\X - \bar\X}^2 \\
 &\quad  + \frac{1}{\mu_t+ L_t} \normfro{\grad \varphi^t(\X)}^2,
 \ead 
 \ee
 where we also used $ \normfro{\grad \varphi^t(\X)}\leq   \normfro{\nabla \varphi^t(\X)}$ by the non-expansiveness of $\p_{\T_{\X}\M^N}$. Substituting \eqref{Rieman restricted strong convexity-1} into \eqref{ineq:RSI second look} and noting \eqref{ineq:bound type 2}, we get
%  \begingroup\allowdisplaybreaks
\begin{align*}
&\quad \inp{\grad \varphi^t(\X)} {\X-  \bar\X} \\
& \geq  \frac{\mu_t L_t}{\mu_t + L_t}(1-\frac{4r}{N}\normfro{\X - \bar\X}^2 - \frac{\mu_t + L_t}{\mu_t L_t}\varphi^t(\X) )\normfro{\X - \bar\X}^2   + \frac{1}{\mu_t + L_t} \normfro{\grad  \varphi^t(\X)}^2.
\end{align*}
% \endgroup
Therefore, when $\normfro{\X-\bar \X} \rightarrow 0 $,  we have $\varphi^t(\X)\rightarrow 0$ by \cref{lem:lipschitz}. We get\[ c_d=   \frac{\mu_t L_t}{\mu_t + L}(1-\frac{4r}{N}\normfro{\X - \bar\X}^2 - \frac{\mu_t + L_t}{\mu_t L_t}\varphi^t(\X) )\rightarrow\frac{\mu_t L_t}{\mu_t + L_t}. \]
By the same arguments as of \cref{thm:linear_rate_consensus}, we get the asymptotic rate being $\frac{L_t - \mu_t}{L_t+\mu_t}$ with  $\alpha = \frac{2}{L_t+\mu_t}$, and $\frac{L_t - \mu_t}{L_t+\mu_t} \leq \sigma_2^t$. Also, using similar arguments as \eqref{ineq:Eucldiean rate with 1},   we can get the rate of $\sigma_2^t$ with $\alpha=1$ as the Euclidean case by noting that \eqref{ineq:error bound} is asymptotically $ \mu_t\normfro{\X-\bar{\X}} \leq \normfro{\grad \varphi^t(\X)}$.   
%%%%%%%%%%%%%%%%%%%%%%%%%%%%%%%%%%%%%%%%%%%%%%%%%%%%%%%%%%%%%%%%%%%%%%%%%%%%%%%%

\section{Numerical experiment}\label{sec:numerical}

We   test the stepsize   on a ring graph.  %i.e., the neighborhoods of $i$ are $i-1$ and $i+1$, and we view node $N+1$ as $1$, node $-1$ as $N$. 
The matrix $W$ is given as follows:
{\small
\begin{align*}W=\left(\begin{array}{cccccc}
1/3 & 1/3 &       &        & 	 & 1/3 \\
1/3 & 1/3 & 1/3   &        & 	 & 	   \\
    & 1/3 & 1/3   & \ddots & 	 &	   \\
    &     &\ddots & \ddots & 1/3 & 	\\
    &     &       &   1/3  & 1/3 & 1/3 \\
1/3 &     &       &    	   & 1/3 & 1/3
\end{array}\right).\end{align*}
}

\begin{figure}[ht]
	\begin{center}
		\minipage{0.23\textwidth}
		\subfigure[$W$: gradient]{
			{\includegraphics[width=0.9\linewidth]{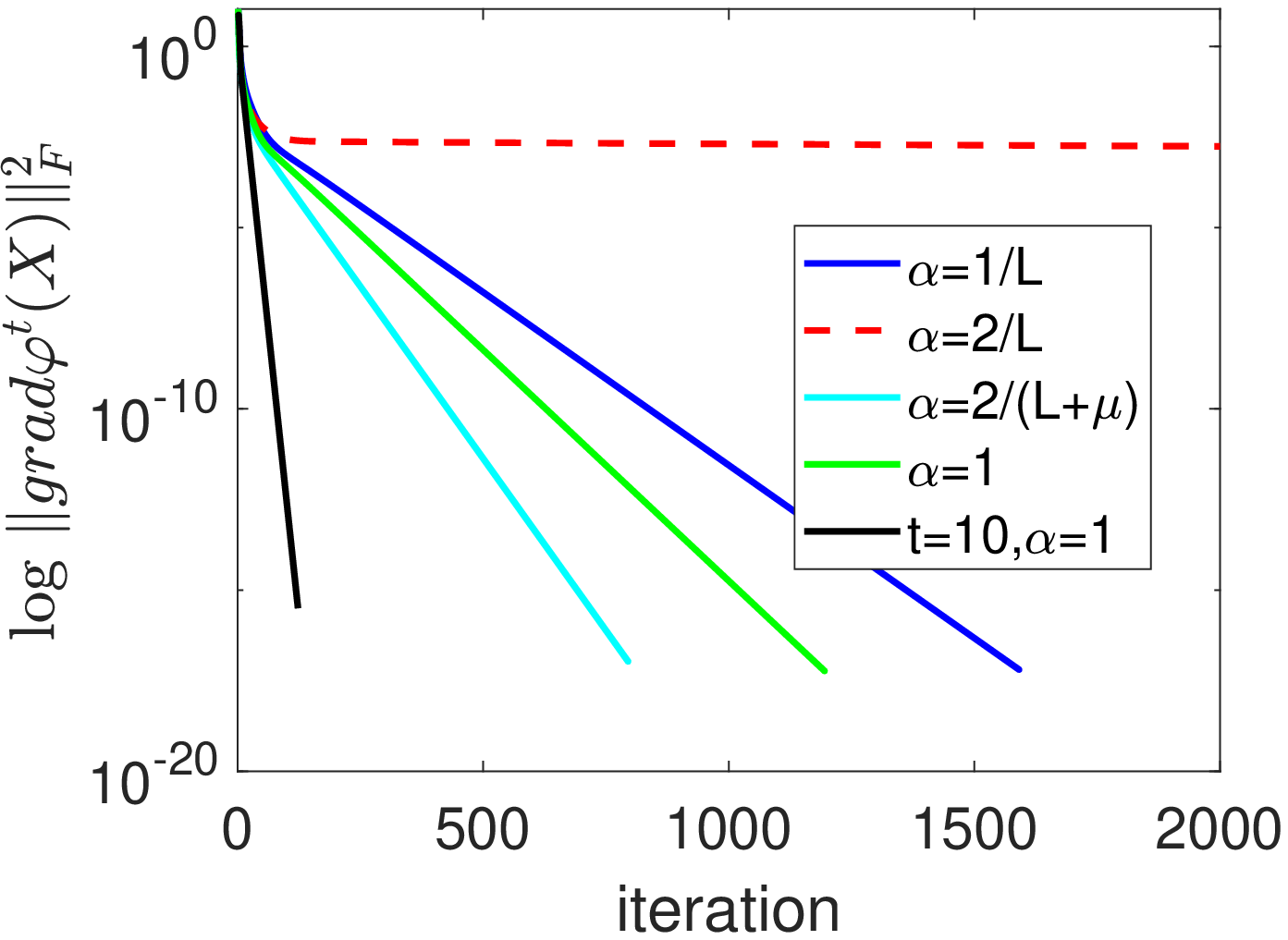}} }
		\endminipage\hfill
		\minipage{0.23\textwidth}
		\subfigure[$W$: distance]{	
			{\includegraphics[width=0.9\linewidth]{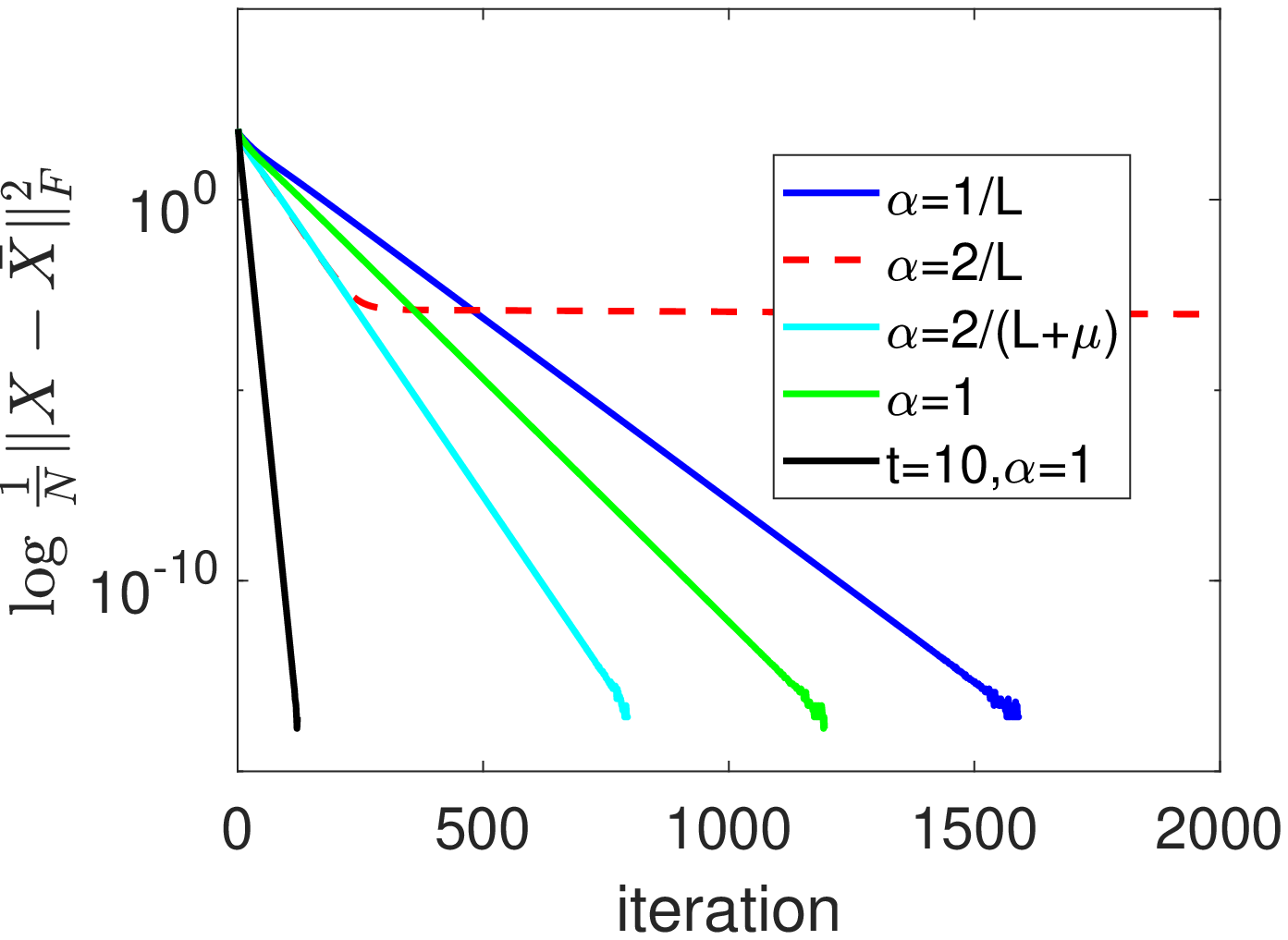}}}
		\endminipage\hfill	
		\minipage{0.23\textwidth}
		\subfigure[$(W+I_N)/2$: gradient]{\includegraphics[width=0.9\linewidth]{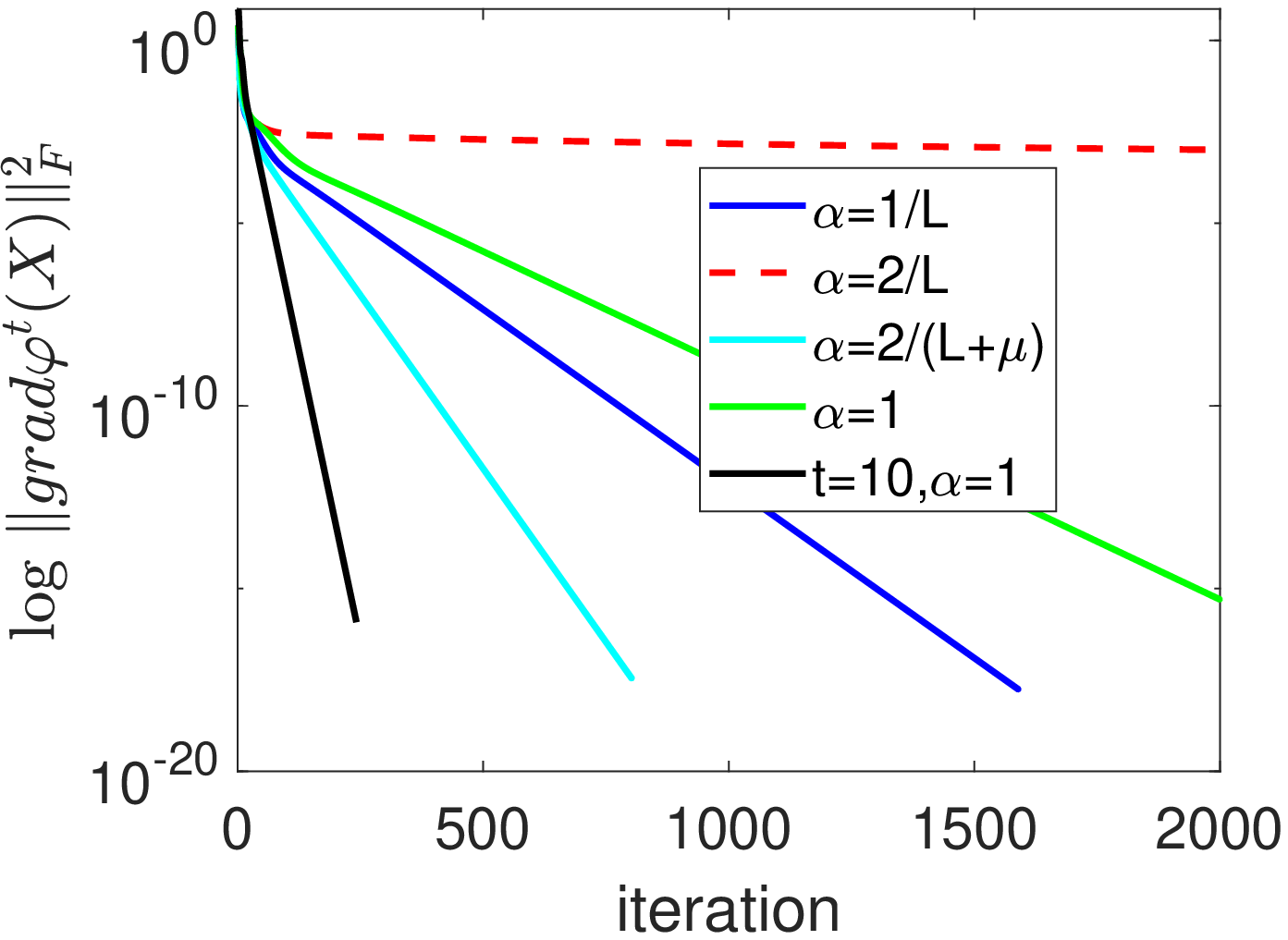} }
				% 	{\includegraphics[width=0.45\linewidth]{./figure/consensus3051-Wpd-dist.eps}}}
		\endminipage\hfill
		\minipage{0.23\textwidth}
		\subfigure[$(W+I_N)/2$: distance ]{\includegraphics[width=0.9\linewidth]{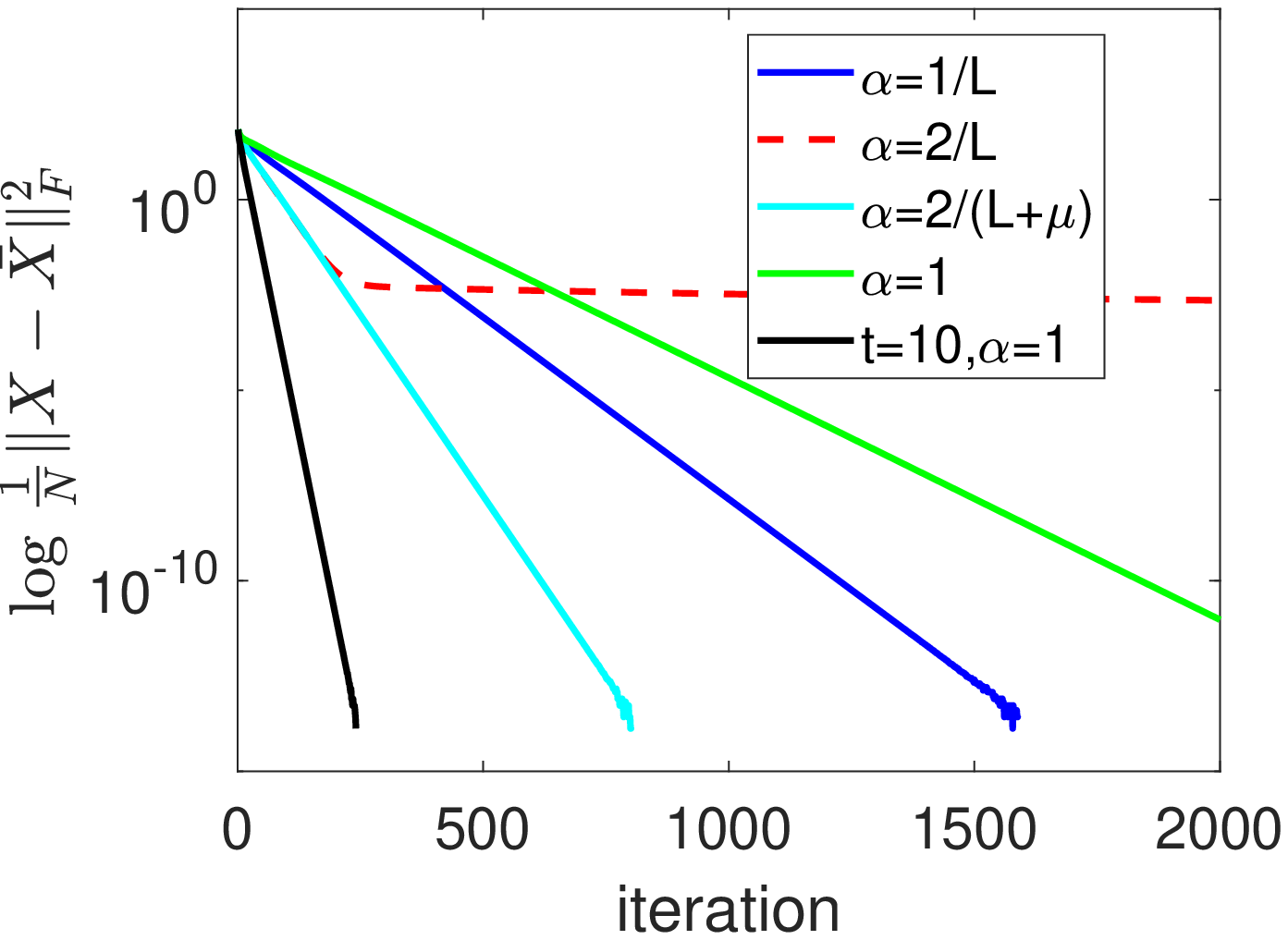}}
		\endminipage\hfill	
%			\minipage{0.24\textwidth}
%		\subfigure[  $\sigma_2^2$ minus slope]{	
%			{\includegraphics[width=0.9\linewidth]{./figure/slope}}}
%		\endminipage\hfill	
		\caption{Numerical results for $N = 30, d =5, r=2$. }\label{figurea}
		\label{figure:drgta}
	\end{center}
	\vskip -0.1in
\end{figure}
 We ran \cref{alg:DRCS} with four choices of stepsize: $1/L, 2/(L+\mu), 2/L, 1$, all of them are stopped when $\frac{1}{N}\normfro{\X_{k}-\bar \X_k}^2\leq  2\times 10^{-16}$. 
 In \cref{figurea} (a)(b), We have $L =  1-\lambda_{\min}  = \frac{4}{3}$.   For \cref{figurea} (c)(d), the doubly stochastic matrix is given by $(W+I_N)/2$  and we have $L =   1-\lambda_{\min}  = \frac{2}{3}$. 
The left column is  log-scale $\normfro{\grad \varphi^t(\X)}^2$ and the right column is log-scale distance $\frac{1}{N}\normfro{\X_{k}-\bar \X_k}^2 $. We see that \cref{alg:DRCS}  with $\alpha=2/L$  does not converge to a critical point. In both cases, $\alpha=2/(\mu+L)$ produces the fastest convergence.  The black line is the convergence of multi-step consensus with $t=10$ and $\alpha=1$ and the rest lines are for $t=1$. The convergence rate is about 10 times of that green line. 
%We also compare the difference between   $\sigma_2^2$ with  the slope. The slope is equal to   the ratio  $\frac{ \normfro{\X_{k+1}-\bar\X_{k+1}}^2}{\normfro{\X_{k}-\bar\X_{k}}^2}$.  
% In \cref{figure:drgta} (e), we plot the log scale of difference between $\sigma_2^2$ and the slope. In all cases, we see that the difference finally converges to approximately $0$.

\section{Conclusion}
In this paper, we provided the global and local convergence analysis of DRCS, a distributed method for consensus on the Stiefel manifold. We showed that the convergence rate asymptotically matches the Euclidean counterpart, which scales with the second largest singular value of the communication matrix. The main technical contribution is to generalize the Euclidean restricted secant inequality to the Riemannian version. In the future work, we would like to study the preservation of iteration in the region $\N_{2,t}$ without multi-step consensus and to estimate the constant $C_{\M,\varphi^t}$ for stepsize.  

\section*{APPENDIX}\label{sec:appendix}

\begin{proof}[\textbf{Proof of inequality \eqref{restricted strong convexity}}]
   Without loss of generality, we assume $d=r=1$.  Let $U_1,U_2,\ldots,U_N$ be the orthonormal eigenvectors of $I_N - W$, corresponding to the eigenvalues $0 = \lambda_1 < \lambda_2 \leq \ldots\leq \lambda_N$. %We assume that $\X\notin  \mathrm{span}\{U_1\}$.
    Then, we have that $\X- \hat \X = \sum_{i=1}^N c_i U_i$. Since $\X-\hat \X$ is orthogonal to $ \mathrm{span}\{U_1\}$, we have $c_1=0$. % and $c_2,\ldots,c_N$ are not all zeros. 
    Note that $\nabla \varphi(\X) = (I_N-W)\X= (I_N-W)(\X-\hat\X)$. We get
   \be\label{eq:Euclidean helpful} \normfro{\X - \hat\X}^2 = \sum_{i=2}^N c_i^2 \quad \text{and} \quad \normfro{\nabla\varphi(\X)}^2 =\sum_{i=2}^N c_i^2 \lambda_i^2.     \ee
     
   Then, \eqref{restricted strong convexity} reads
   \begingroup\allowdisplaybreaks
   \begin{align}
     &\quad \inp{\X - \hat\X}{\nabla \varphi (\X)}  =\inp{\X - \hat\X}{(I_N - W) (\X - \hat\X)}\notag \\
    &~~~~~~~~~~ = \inp{\sum_{i=2}^N c_i U_i}{\sum_{i=2}^N c_i \lambda_i U_i}\notag \\
      &~~~~~~~~~~ = \sum_{i=2}^N c_i^2 \lambda_i \geq \frac{1}{L+\mu}  \sum_{i=2}^N ( \mu Lc_i^2 +  c_i^2 \lambda_i^2 )\notag \\
   &~~~~~~~~~~ = \frac{\mu L}{\mu + L}\normfro{\X - \hat\X}^2 + \frac{1}{\mu + L} \normfro{\nabla \varphi(\X)}^2,\label{RSI Euclidean proof}
   \end{align}
   \endgroup
   where the inequality follows since $\mu = \lambda_2$ and $L=\lambda_N$.  
\end{proof}

	\begin{proof}[\textbf{Proof of linear rate of PGD with $\alpha_e=1$}]
	Firstly, one can easily verify $L\normfro{\X - \hat \X} \geq \normfro{\nabla \varphi(\X)} \geq \mu \normfro{\X - \hat \X}$ using \eqref{eq:Euclidean helpful}. 
	We have
 \be\label{ineq:Eucldiean rate with 1}
 \begin{aligned}
      &\quad \normfro{\X_{k+1} - \hat{\X}_{k+1}}^2\leq \normfro{\X_{k+1} - \hat{\X}_{k}}^2\\
      &\leq \normfro{\X_{k} - \hat{\X}_{k}}^2 + \normfro{\nabla\varphi(\X_{k})}^2 - 2\inp{\nabla \varphi(\X_{k})}{\X_{k}-\hat\X_{k}} \\
      & \stackrel{\eqref{restricted strong convexity}}{\leq}(1-\frac{2\mu L}{\mu + L})\normfro{\X_{k} - \hat{\X}_{k}}^2 + (1- \frac{2}{\mu+L})  \normfro{\nabla\varphi(\X_{k})}^2.
 \end{aligned}\ee
 If $\frac{2}{\mu+L}\geq 1$, i.e., $\lambda_2(W)+\lambda_N(W)\geq 0$, this implies $\sigma_2 = \lambda_2(W).$ Combining $\normfro{\nabla \varphi(\X)} \geq \mu \normfro{\X - \hat \X}$ with \eqref{ineq:Eucldiean rate with 1} yields 
  \begin{align*}
      &\quad \normfro{\X_{k+1} - \hat{\X}_{k}}^2\\
      & \leq (1-\frac{2\mu L}{\mu + L} - \mu^2 + \frac{2\mu^2}{L+\mu} )\normfro{\X_{k} - \hat{\X}_{k} }^2\\
      & = (1-\mu)^2 \normfro{\X_{k} - \hat{\X}_{k} }^2=\sigma_2^2 \normfro{\X_{k} - \hat{\X}_{k} }^2.
 \end{align*}
 If $\frac{2}{\mu+L}< 1$, then $\lambda_2(W)+\lambda_N(W)< 0$, this implies $\sigma_2 = -\lambda_N(W).$ Combining $\normfro{\nabla \varphi(\X)} \leq L \normfro{\X - \hat \X}$ with   \eqref{ineq:Eucldiean rate with 1} implies
  \begin{align*}
      &\quad \normfro{\X_{k+1} - \hat{\X}_{k}}^2\\
      & \leq (1-\frac{2\mu L}{\mu + L} - L^2 + \frac{2L^2}{L+\mu} )\normfro{\X_{k} - \hat{\X}_{k} }^2\\
      & =  (1-L)^2 \normfro{\X_{k} - \hat{\X}_{k} }^2 = \sigma_2^2 \normfro{\X_{k} - \hat{\X}_{k} }^2.
 \end{align*}
 \end{proof}

\begin{proof}[\textbf{Proof of Lemma \ref{lem:distance between Euclideanmean and IAM}}]
	
	Note that
	\be\label{ineq:reg_condition-2-1}
	\bad
	\normfro{\X-\hat \X}^2 &= 	\sum_{i=1}^N \normfro{x_{i}- \hat x}^2=  N (r - \normfro{\hat x}^2) \\
	& = N (\sqrt{r}+ \normfro{\hat x})( \sqrt{r}-  \normfro{\hat x})\\
	&\leq 2N (  {r}-  \sqrt{r}  \normfro{\hat x}),
	\ead \ee
	where the inequality is due to $ \normfro{ \hat x}\leq  \sqrt{r}$. Since
	\be\label{iam}\bar{x} = \p_{\St}(\hat x) = uv^\top,\ee where $usv^\top = \hat x$ is the singular value decomposition, we get
	\be\label{equation:l2 distance X barX}\bad \normfro{\X-\bar{\X}}^2&= \sum_{i=1}^N(2r -2\inp{x_i}{\bar x})\\
	& = 2N(r-\inp{\hat x}{\bar x})= 2N(r -   \|\hat x\|_{*}), \ead \ee
	where $\|\cdot\|_*$ is the trace norm.  Let $\hat \sigma_1\geq \ldots\geq \hat \sigma_r$ be the singular values of $\hat x$.  It is clear that $\hat\sigma_1\leq 1$ since   $\normtwo{\hat x}\leq \frac{1}{N} \sum_{i=1}^N \normtwo{x_i}\leq 1$. The inequality $\|\hat x\|_* = \sum_{i=1}^r \hat{\sigma}_i\leq \sqrt{r}\sqrt{\sum_{i=1}^r \hat \sigma_i^2} =\sqrt{r}\normfro{\hat x} $, together with \eqref{ineq:reg_condition-2-1} and \eqref{equation:l2 distance X barX} imply that
	\[ \normfro{\X - \hat \X}^2 \leq \normfro{\X - \bar \X}^2. \]
	Next, we also have $\|\hat x\|_* = \sum_{i=1}^r \hat{\sigma}_i \geq  \sum_{i=1}^r \hat{\sigma}_i^2 =  \normfro{\hat x}^2$. This yields 
	\[ \frac{1}{2}\normfro{\X-\bar \X}^2 =  N (r -  \|\hat x\|_{*}) \leq  N (r -  \normfro{\hat x}^2) =\normfro{\X-\hat \X}^2,  \]
	which proves \eqref{ineq:relation_eum_manifmean}.

	By utilizing the fact $	\normfro{\X-\hat \X}\leq \normfro{\X-\bar{\X}}$, we have
	\be \label{ineq:reg_condition-2-2}
	\bad
	\sqrt{r} \sqrt{\sum_{i=1}^r \hat\sigma_i^2} =\sqrt{r}	\normfro{\hat x} \geq \normfro{\hat x}^2  =  r-   \frac{1}{N} \normfro{\hat{\X}-\X}^2\geq  r-   \frac{1}{N} \normfro{\bar{\X}-\X}^2, 
	\ead\ee 
	where we used $\normfro{\hat x}=\normfro{\frac{1}{N}\sum_{i=1}^{N}x_i}\leq \sqrt{r}.$ If $\normfro{\X - \bar \X}^2\leq  N/2$ (by assumption), we can square both sides of above and note $\hat\sigma_i^2\leq 1$ for $i\in [r-1]$ to get
	\[ \hat \sigma_r^2\geq 1- 2\frac{\normfro{\X-\bar\X}^2}{N} + \frac{\normfro{\X-\bar\X}^4}{N^2r}\geq  1- 2\frac{\normfro{\X-\bar\X}^2}{N}. \]
	Then, we have \be\label{ineq:reg_condition-2-3} \hat \sigma_r \geq  \sqrt{1-2\frac{\normfro{\X-\bar\X}^2}{N}  } \geq 1-2 \frac{\normfro{\X-\bar\X}^2}{N},   \ee
	where we use $\sqrt{1-s}\geq  1- s$ for any $1\geq s\geq 0$.
	Recall that $\bar{x} = \p_{\St}(\hat x) = uv^\top$.
	Hence, it follows that 
	\begin{align*}
	&\quad\normfro{\hat{x} - \bar x}^2  = r - 2\inp{ \hat{x}}{ \bar x}+ \normfro{\hat{x}}^2 \\&= r - 2\sum_{i=1}^{r}\hat \sigma_i + \sum_{i=1}^{r}\hat \sigma_i^2 = \sum_{i=1}^r (1-\hat\sigma_i)^2  \leq \frac{ 4r\normfro{\X-\bar\X}^4}{N^2}.
	\end{align*}
	Hence, we have proved \eqref{key}. 
 Finally, 
 \begin{align*}
 	&\quad\normfro{\X - \hat \X}^2 =  \sum_{i=1}^N \inp{x_i - \hat x}{x_i - \hat x}\\
 	&=  \sum_{i=1}^N \inp{x_i - \hat x}{x_i - \bar x} +  \sum_{i=1}^N \inp{x_i - \hat x}{\bar x- \hat x} \\
 	& =   \normfro{\X-\bar\X}^2 + \sum_{i=1}^N \inp{\bar{x}- \hat x }{x_i - \bar x  }\notag\\
% 	& = \normfro{\X-\bar\X}^2 + N \inp{\bar{x}-\hat x }{\hat{x} - \bar x  }\notag\\
 	& = \normfro{\X-\bar\X}^2 - N \normfro{\bar{x}-\hat x }^2\notag\\
 	&\stackrel{\eqref{key}}{\geq } \normfro{\X-\bar\X}^2  -  \frac{4r \normfro{ \X-\bar\X}^4}{ {N}}\notag,
 	%	&\stackrel{\eqref{key-1}}{\geq }  \normfro{\X-\bar\X}^2  - \normfro{\X-\bar\X}\cdot\frac{2 \normfro{ \X-\bar\X}^2}{\sqrt{N}}\notag\\
% 	&\stackrel{\eqref{def:region quadratic growth}}{\geq } (1-   4r\delta_t^2  )\normfro{\X-\bar\X}^2.
 \end{align*}
 where we used $\sum_{i=1}^N \inp{x_i - \hat x}{\bar x- \hat x}=0$ in the third line. 
\end{proof}

\begin{proof}[\textbf{Proof of \cref{lem:relation between R and E}}]
	
	It follows that
	\begin{align}
	&   \inp{\grad \varphi^t(\X)}{\Y - \X}  \notag \\
	=&   \inp{\nabla \varphi^t(\X)}{\p_{\T_{\X}\M^N}(\Y - \X)}\notag \\
	=&  \inp{\nabla \varphi^t(\X)}{\Y-\X}\notag  -  \sum_{i=1}^N\inp{ \nabla \varphi^t_i(\X)}{\p_{N_{x_i}\M} (y_i - x_i) } \notag\\
	=&  \inp{\nabla \varphi^t(\X)}{\Y-\X}+ \frac{1}{4}\sum_{i=1}^N\inp{   \nabla \varphi^t_i(\X)^\top x_i + x_i^\top \nabla \varphi^t_i(\X) }{    (y_i-x_i)^\top (y_i-x_i) }\notag.
	\end{align}	
	
%	Utilizing the symmetry of $(y_i-x_i)^\top (y_i-x_i) $, we have
%	$$ \bad & \sum_{i=1}^N\inp{ \nabla \varphi^t_i(\X)}{\half x_i (y_i-x_i)^\top (y_i-x_i) }\\
%$$
	%=&  \frac{1}{4}\sum_{i=1}^N\Tr(  (y_i-x_i)[ \nabla \varphi^t(x_i)^\top x_i + x_i^\top \nabla \varphi^t(x_i)] (y_i-x_i)^\top ). 
	Since	\[ \frac{1}{2}[ \nabla \varphi^t_i(\X)^\top x_i + x_i^\top \nabla \varphi^t_i(\X)] = \half \sum_{j=1}^N W^t_{ij} (x_i - x_j)^\top (x_i - x_j)   \]
	  is positive semi-definite,
	we get 
	% all of its eigenvalues lie in $[0, 2-2W^t_{ii}]$.  Hence, we have 
	\be\label{compareL} \bad  
	%& \sum_{i=1}^N  (1-W^t_{ii})\normfro{y_i-x_i}^2\\\geq & 
	\sum_{i=1}^N\inp{ \nabla \varphi^t_i(\X)}{\half x_i (y_i-x_i)^\top (y_i-x_i) }
	\geq   0. \ead\ee 
 Therefore, we get
	\begin{align}
	%   &\quad \inp{\grad \varphi^t(\X)}{\Y - \X} -  \sum_{i=1}^N   (1-W^t_{ii})\normfro{y_i-x_i}^2 \notag\\
	%	&\leq  
	\inp{\nabla \varphi^t(\X)}{\Y-\X}% \notag \\
	&\leq  \inp{\grad \varphi^t(\X)}{\Y - \X} .
	\end{align}
\end{proof}

 \begin{proof}[\textbf{Proof of Lemma \ref{lem:lipschitz}}]
% 	The proof is similar as that of \cite[Theorem 1]{li2019nonsmooth}. 
The largest eigenvalue of  $\nabla^2 \varphi^t(\X) =  (I_N-W^t)\otimes I_d $ is $L_\phi=  1- \lambda_{N}(W^t) $  in Euclidean space, where $\lambda_{N}(W^t)$ denotes the smallest eigenvalue of $W^t$.  For any $\X,\Y\in(\R^{d\times r})^N$, it follows that\cite{nesterov2013introductory}
	\be\label{ineq:lip_smooth_E}
	\varphi^t(\Y) -  \left[ \varphi^t(\X) + \inp{\nabla \varphi^t(\X)}{\Y-\X} \right]  \leq \frac{L_\phi}{2}\normfro{\Y-\X}^2.
	\ee	
%	Let $N_x\M$ be the normal space of $\St(d,r)$ at $x$. 
%It follows from \eqref{Riemannain grad converges to Euclidgrad} that
%	\begin{align}
%	&\quad \inp{\nabla \varphi^t(\X)}{\Y-\X}\notag \\
%%	&=  \sum_{i=1}^N \left( \inp{\grad \varphi^t(x_i)}{y_i-x_i} +  \inp{\p_{N_{x_i}}\nabla \varphi^t(x_i)}{y_i-x_i} \right) \notag \\
%%	& =  \inp{\grad \varphi^t(\X)}{\Y - \X} + \sum_{i=1}^N \inp{ \nabla \varphi^t(x_i)}{\p_{N_{x_i}}(y_i-x_i)}\label{lips-proof-00} \\
%	&=   \inp{\grad \varphi^t(\X)}{\Y - \X} \notag \\	&\quad -  \sum_{i=1}^N\inp{ \nabla \varphi^t(x_i)}{\half x_i (y_i-x_i)^\top (y_i-x_i) }. \label{lips-proof-01}
%	\end{align}	
%		where \eqref{lips-proof-00} follows from the symmetric property of   projection operator  $\p_{N_{x_i}}$ and \eqref{lips-proof-01} holds since \eqref{property of projection onto tangent}.
%		\begin{align*} \p_{N_{x_i}}(y_i-x_i) &= \half x_i\left( (y_i-x_i)^\top x_i + x_i^\top (y_i-x_i) \right)\\
%		& = -\half x_i (y_i-x_i)^\top (y_i-x_i),  \end{align*}
%		since $x_i^\top x_i=y_i^\top y_i=I_r$. 
%Utilizing the symmetry of $(y_i-x_i)^\top (y_i-x_i) $, we have
%  	$$ \bad & \sum_{i=1}^N\inp{ \nabla \varphi^t(x_i)}{\half x_i (y_i-x_i)^\top (y_i-x_i) }\\
%  = & \sum_{i=1}^N\inp{ \frac{1}{2}[ \nabla \varphi^t(x_i)^\top x_i + x_i^\top \nabla \varphi^t(x_i)] }{  \frac{1}{2} (y_i-x_i)^\top (y_i-x_i) }\\
%  =&  \frac{1}{4}\sum_{i=1}^N\Tr(  (y_i-x_i)[ \nabla \varphi^t(x_i)^\top x_i + x_i^\top \nabla \varphi^t(x_i)] (y_i-x_i)^\top ). \ead $$
Together with \eqref{ineq:key relation between R and E}, this implies that
	\be\label{ineq:lip_smooth_R}
	\bad
%&\quad	-\sum_{i=1}^{N}(1+\frac{L_\phi}{2}-W_{ii}^t)\normfro{y_i - x_i}^2\\&\leq  
\varphi^t(\Y) - \left[ \varphi^t(\X) + \inp{\grad  \varphi^t(\X)}{\Y-\X} \right]  
\leq \frac{L_\phi}{2}\normfro{\X-\Y}^2.
\ead\ee	
%	Similarly, we also have 
%	\be\label{ineq:lip_smooth_R-1}
%\bad
%&\quad \varphi^t(\X) - \left[ \varphi^t(\Y) + \inp{\grad  \varphi^t(\Y)}{\X-\Y} \right]  \\
%&\leq \sum_{i=1}^{N}(1-W^t_{ii})\normfro{y_i - x_i}^2.
%\ead\ee	
	%Adding \eqref{ineq:lip_smooth_R} and \eqref{ineq:lip_smooth_R-1} gives \eqref{ineq:lips_riemanniangrad}. 
%	Moreover, if  $\frac{1}{2} \sum_{j=1}^N W^t_{ij} [ x_j^\top x_i + x_i^\top x_j ]$ and  $\frac{1}{2} \sum_{j=1}^N W^t_{ij} [ y_j^\top y_i + y_i^\top y_j ]$ are positive semidefinite for all $i\in[N]$, we have
%	$$ \bad & \sum_{i=1}^N\inp{ \nabla \varphi^t(x_i)}{\half x_i (y_i-x_i)^\top (y_i-x_i) }
%\leq  \half\normfro{\X-\Y}^2. \ead $$
%	This gives us		\be\label{ineq:lip_smooth_R-2}
%	\bad
%	 \varphi^t(\Y) - \left[ \varphi^t(\X) + \inp{\grad  \varphi^t(\X)}{\Y-\X} \right]  \leq \half \normfro{\Y-\X}^2.
%	\ead\ee	 
The proof is completed. 
\end{proof}

\begin{proof}[\textbf{Proof of \cref{lem:sublinear-convergence}}]
	The proof follows \cite[Theorem 3]{Liu-So-Wu-2018}. We only need to verify the following three properties:
	\begin{enumerate}
		\item [(A1).] (Sufficient descent) There exists a constant $\kappa>0$ and sufficiently large $K_1$ such that for $k\geq K_1$, {\small
		\begin{align*}
		\varphi^t(\X_{k+1})-\varphi^t(\X_k)
		\leq  -\kappa \normfro{\grad \varphi^t(\X_k)}\cdot\normfro{\X_k - \X_{k+1}}.
		\end{align*} }
		\item [(A2).] (Stationarity) There exists an index $K_2 > 0$ such that for $k\geq K_2$,
		\[ \normfro{\grad \varphi^t(\X_k)} = 0 \Rightarrow  \X_k = \X_{k+1}. \]
		\item[(A3).] (Safeguard) There exist a constant $C_3>0$ and an index $K_3> 0$  such that for $k \geq K_3$
		\[ \normfro{\grad \varphi^t(\X_k) }\leq C_3\normfro{\X_{k}-\X_{k+1}}. \]
	\end{enumerate}
 The main difference is that we use \cref{lem:lipschitz} to derive the sufficient descent property (A1). 
	Let us first consider (A1).
	Using \eqref{ineq:lipschitz} of \cref{lem:lipschitz}, one has 
	\begin{align*}
	\quad &\varphi^t(\X_{k+1})\leq  \varphi^t(\X_k) + \inp{\grad \varphi^t(\X_k)}{\X_{k+1} - \X_k} + \frac{L_t}{2}\normfro{\X_k - \X_{k+1}}^2.
	\end{align*}
	Let us start with the following
	\begin{align*}
	&\quad \inp{\grad \varphi^t(\X)}{\X_{k+1} - \X_k} \\
	&= \sum_{i=1}^N  \inp{\grad \varphi^t_i(\X_k) }{x_{i,k+1} - x_{i,k}}\\
	& =  \sum_{i=1}^N  \inp{\grad \varphi^t_i(\X_k) }{\Retr_{x_{i,k}}(-\alpha \grad \varphi^t_i(\X_k) )  - x_{i,k}}\\
	&\stackrel{\eqref{ineq:ret_second-order}}{\leq} (M\alpha^2\cdot \normfro{\grad \varphi^t(\X_k) }-\alpha)\normfro{\grad \varphi^t(\X_k) }^2  
	\end{align*}
	and 
	$
	\normfro{\X_{k+1}-\X_k}^2 \stackrel{\eqref{ineq:ret_nonexpansive}}{\leq } \alpha^2\normfro{\grad \varphi^t(\X_k) }^2 .
	$
	We now get
	\begin{align*}
	 \varphi^t(\X_{k+1})\leq  \varphi^t(\X_k)  + [  (MG_k+\frac{L_t}{2})\alpha^2-\alpha]\normfro{\grad \varphi^t(\X_k) }^2,
	\end{align*}
	where $G_k = \normfro{\grad \varphi^t(\X_k)}$.
	Therefore, for any $\beta\in(0,1)$, if $\alpha < \bar{\alpha}_k:=\frac{1-\beta}{MG_k + L_t/2}$, we have
	\be\label{ineq:suff-decrease}
	\varphi^t(\X_{k+1})
	\leq  \varphi^t(\X_k)  -\alpha\beta\normfro{\grad \varphi^t(\X_k) }^2.
	\ee
	Note that $\bar{\alpha}_k\geq \frac{1-\beta}{MG + L_t/2}$, the stepsize $\alpha < \bar{\alpha}_k$ is well defined.
	Again, by
	$
	\normfro{\X_{k+1}-\X_k}^2 \stackrel{\eqref{ineq:ret_nonexpansive}}{\leq } \alpha^2\normfro{\grad \varphi^t(\X_k) }^2,
	$
% 	Using \eqref{ineq:ret_second-order} also yields {\rd we need G below after M? \& N}
% 	\begin{align*} \normfro{\X_{k}-\X_{k+1}} &\leq ( M \alpha^2 + \alpha)\normfro{\grad\varphi^t(\X_k) }\\
% 	&\leq ( M   + 1)\alpha\normfro{\grad\varphi^t(\X_k) },  \end{align*}
% 	{\rd where we use $\alpha\leq 1$ since $L_t\geq1 $.}
	 we get the sufficient decrease condition in (A1) for any $k\geq 0$ with $\kappa = \beta$
	\be\label{ineq:suff-decrease-1}
	\bad
	 \varphi^t(\X_{k+1})\leq  \varphi^t(\X_k)  -\beta\normfro{\grad \varphi^t(\X_k) } \cdot\normfro{\X_{k}-\X_{k+1}}.
	\ead\ee
	The condition (A2) is automatically satisfied by the iteration of \cref{alg:DRCS}. For (A3), the argument is the same as that of \cite[Theorem 3]{Liu-So-Wu-2018}. By \eqref{ineq:suff-decrease}, we have  $\sum_{k=0}^\infty\alpha \normfro{\grad\varphi^t(\X_k) }^2 \leq \varphi^t(\X_0) - \inf\varphi^t(\X)<\infty$, which implies 
	\[ \lim_{k\rightarrow \infty} \alpha \normfro{\grad\varphi^t(\X_k) }^2 = 0.  \]
	So, there exists $K_3> 0$ such that $\normfro{\grad\varphi^t(\X_k) }$ is sufficiently small  whenever $\alpha >0$. Using the second-order property of retraction $\Retr_x(\xi) = x+\xi + \mathcal{O}(\normfro{\xi}^2)$, we have the property (A3). 
	
	By \cite[Theorem 2.3]{schneider2015convergence}, (A1)-(A2) together with \eqref{L-inequality} imply the convergence to a critical point. With (A3), one has that the convergence rate is sub-linearly if $\theta<1/2$ and linearly if $\theta=1/2$, respectively.  
\end{proof}

\begin{proof}[\textbf{Proof of Proposition \ref{prop:local opt}}]
 Let $B: = W\otimes I_d$. The necessity is trivial by letting $y=[B\X]_i$ if $x_1=x_2=\ldots=x_N$. Now, if $\X$ is a first-order critical point, then it follows from \cref{prop:optcond} that
	\begin{align*} 
	&\grad \varphi^t_i(\X)   =  \nabla \varphi^t_i(\X) - \frac{1}{2}x_i(x_i^\top \nabla \varphi^t_i(\X) +\nabla \varphi^t_i(\X)^\top x_i)\\
	& = (I_d - \frac{1}{2}x_ix_i^\top)( \nabla \varphi^t_i(\X) -x_i\nabla \varphi^t_i(\X)^\top x_i) = 0, \quad \forall i \in[N].
	\end{align*}
	Note that since $I_d - \frac{1}{2}x_ix_i^\top$ is invertible, one has
	\be\label{first_order:condition}  [B\X]_i- x_i ( [B\X]_i^\top x_i)=0, \quad \forall i\in[N].\ee
	Multiplying both sides by $x_i^\top$ yields
	\be\label{first_order:condition1}  x_i^\top [B\X]_i =   [B\X]_i^\top x_i , \quad \forall i\in[N].\ee
    For the sufficiency, 
	let $\Gamma_i :=\sum_{j=1}^N W_{ij}(x_j^\top x_i)$, $i\in[N].$  
	From  \eqref{first_order:condition}, we get 
	\be\label{eq:prop:second_is_global-derivation0}  x_i \Gamma_i =  \sum_{j=1}^N W_{ij} x_j, \quad \forall i\in[N].\ee
	Summing above over $i\in[N]$ yields
	$\sum_{i=1}^N  x_i\Gamma_i = \sum_{i=1}^N x_i$. 
	Taking inner product with $y$ on both sides gives	$\sum_{i=1}^N\inp{y}{x_i(I_r-\Gamma_i)}  =0$. Note that $I_r-\Gamma_i$ is symmetric for all $i$ due to \eqref{first_order:condition1}. It is also positive semi-definite. 
	Since $\inp{x_i}{y}> r-1$ for all $i$, we get that $\Omega_i : = \frac{1}{2}(x_i^\top y + y^\top x_i)$ is positive definite.
Then, it follows that
	\begin{align*}
	  \inp{y}{x_i(I_r-\Gamma_i)} 
	   = \Tr(\Omega_i^{1/2}(I_r-\Gamma_i) \Omega_i^{1/2} ) \geq 0.
	\end{align*}
The equation  $\sum_{i=1}^N\inp{y}{x_i(I_r-\Gamma_i)}   =0$ suggests that $I_r = \Gamma_i$,
  which also implies $x_1=x_2=\ldots=x_N$  by \eqref{eq:prop:second_is_global-derivation0}. 
	
	Furthermore, suppose $y=\bar x$ which is the IAM of $\X$. The condition $\distinfty(\X,\cX^*)<\sqrt{2}$ means that 
	$\normfro{\bar x- x_i}^2< 2$, or equivalently, $\inp{y}{x_i}>r-1$ for all $i\in[N]$.  
\end{proof}

\begin{proof}[\textbf{Proof of \cref{lem:stay in L sphere}}]
	%	Since \begin{align*}
	%		& x_{i,k+1} \\
	%		=& (x_{i,k}- \alpha \grad \varphi_i(\X_k))(I_r + \alpha^2 \grad \varphi_i(\X_k) ^\top \grad \varphi_i(\X_k))^{-1/2}.
	%	\end{align*}
	We prove it by induction. Suppose \eqref{ineq:lower bound inner product} holds for some $k$. 
	For $k+1$, we first have 
	%	\begin{align*}
	%		\inp{x_{i,k}  - \alpha \grad \varphi_i(\X_k)  }{y} \geq  \delta, 
	%	\end{align*}
	%	which reads as
	\begin{align*}
		&  	\inp{x_{i,k}  - \alpha \grad \varphi_i(\X_k)  }{y}\\
		=	& \langle x_{i,k} -  \frac{\alpha}{2} x_{i,k} \sum_{j=1}^N W_{ij}  ( x_{j,k}^\top x_{i,k} +x_{i,k}^\top x_{j,k} ),y \rangle+ \alpha\sum_{j=1}^N W_{ij}  	\inp{	x_{j,k} }{y}   \\
		=&   \frac{\alpha}{2}      \sum_{j=1}^N W_{ij}  \normfro{ x_{i,k}-x_{j,k}}^2 \cdot   \inp{ x_{i,k}}{y}  +(1-\alpha)\inp{ x_{i,k}}{y}  + \alpha\sum_{j=1}^N W_{ij}  	\inp{	x_{j,k} }{y}   \\
		\geq & \delta \frac{\alpha^2 }{2} \sum_{j=1}^N W_{ij} \normfro{x_{i,k} - x_{j,k}}^2 +\delta.
		%	 	\geq  &  \frac{\alpha}{4} \inp{     \sum_{j=1}^N W_{ij}  ( x_{i,k}-x_{j,k})^\top(x_{i,k} - x_{j,k} )}{ x_{i,k}^\top y   +y^\top x_{i,k} } \\ &\quad + \delta \\
		%	\geq & \epsilon \frac{\alpha}{2} \sum_{j=1}^N W_{ij}\normfro{x_{i,k}-x_{j,k}}^2 + \delta,
	\end{align*}
The last inequality follows from $\alpha \leq 1$. Then, since $x_{i,k+1} = \frac{ x_{i,k}  - \alpha \grad \varphi_i(\X_k)}{\sqrt{1+ \alpha^2\normfro{\grad \varphi_i(\X_k)}^2 }}$ (due to \eqref{eq:polar}), we get
	\begingroup\allowdisplaybreaks
	\begin{align}
		\inp{x_{i,k+1} }{y}= & \frac{\inp{x_{i,k}  - \alpha \grad \varphi_i(\X_k)  }{y}}{\sqrt{1+ \alpha^2\normfro{\grad \varphi_i(\X_k)}^2 }}\notag\\
		\geq & \frac{\inp{x_{i,k}  - \alpha \grad \varphi_i(\X_k)  }{y}}{ 1+ \frac{\alpha^2}{2} \normfro{\grad \varphi_i(\X_k)}^2 }\label{implicit1}\\
		\geq &  \frac{\inp{x_{i,k}  - \alpha \grad \varphi_i(\X_k)  }{y}}{ 1+  \frac{\alpha^2}{2} \sum_{j=1}^N W_{ij}\normfro{x_{i,k}-x_{j,k}}^2 }\label{implicit2} \geq \delta,		 
	\end{align}\endgroup
	where we used $\sqrt{1+z^2}\leq 1+\frac{1}{2}z^2$ for any $z\geq 0$ in \eqref{implicit1} and $ \normfro{\grad \varphi_i(\X_k)}^2\leq  \normfro{\nabla \varphi_i(\X_k)}^2 \leq \sum_{j=1}^N W_{ij}\normfro{x_{i,k}-x_{j,k}}^2 $ in \eqref{implicit2}. 
	%By the property of projection operator $\p_{\T_{x\M}}$, one has
	%\begin{align*}
	%	 &\quad \normtwo{ \grad \varphi_i(\X_k) ^\top \grad \varphi_i(\X_k) }\\
	%	 &= \normtwo{\nabla \varphi_i(\X) ^\top \p_{T_{x_i}\M} \nabla \varphi_i(\X)}\\
	%	 &\leq \normtwo{\nabla \varphi_i(\X) ^\top  \nabla \varphi_i(\X)} \\
	%	 &\leq \normtwo{\nabla \varphi_i(\X)  }^2\\
	%	 & \leq \sum_{j=1}^N W_{ij} \normtwo{ x_i - x_j  }^2,
	% \end{align*}
	%where we use $\normtwo{\cdot}\leq \normfro{\cdot}$.
	%Therefore, 
	%\begin{align*}
	%  &\quad (I_r + \alpha^2 \grad \varphi_i(\X_k) ^\top \grad \varphi_i(\X_k))^{1/2}\\
	%    &\preceq  (1+  \alpha^2  \sum_{j=1}^N W_{ij} \normtwo{ x_i - x_j  }^2)^{1/2} I_r\\
	%  &\preceq  (1+ \frac{\alpha^2} {2} \sum_{j=1}^N W_{ij} \normtwo{ x_i - x_j  }^2) I_r,
	%\end{align*}
	%where we use the fact $\sqrt{1+t}\leq 1+\frac{t}{2}$ for any $t\geq0$.
	%We get 
	%	\begin{align*}
	%	&\inp{x_{i,k+1}}{y}	\\
	%	\geq &  	\inp{x_{i,k}  - \alpha \grad \varphi_i(\X_k)  }{y}\\
	%	\geq & \epsilon \frac{\alpha}{2} \sum_{j=1}^N W_{ij}\normfro{x_{i,k}-x_{j,k}}^2 + \delta.
	%\end{align*}
\end{proof}

\begin{proof}[\textbf{Proof of \cref{prop:quadratic growth}}]

	We rewrite the objective $\varphi^t(\X)$ as follows	
	\begingroup\allowdisplaybreaks
	\begin{align}\label{eq:function and gradient}
	& 2\varphi^t(\X)  =  \sum_{i=1}^N \normfro{x_i}^2-  \sum_{i=1,j=1}^N W^t_{ij}\inp{x_i}{x_j}\notag\\
	=&  \sum_{i=1}^N \langle x_i,x_i - \sum_{j=1}^N W^t_{ij}x_j\rangle\notag\\
	= & \inp{\nabla \varphi^t(\X)}{ \X }.
% 	= & \sum_{i=1}^N \langle x_i - \hat x,\hat x - \sum_{j=1}^N W^t_{ij}x_j\rangle + \sum_{i=1}^N \inp{x_i - \hat x}{x_i - \hat x}  \notag\\%\label{proof-quadratic-growth-0}\\
% 	=& \sum_{i=1}^N  \langle x_i - \hat{x},\sum_{j=1}^N ( \frac{1}{N}- W_{ij}^t  )  (x_{j}-\hat x) \rangle + \normfro{\X - \hat \X}^2 \notag\\ 
% 	\geq & (1-\sigma_2^t)\normfro{\X-\hat\X}^2,
	\end{align}\endgroup
	Note that  $ \inp{\nabla \varphi^t(\X)}{ \hat\X  } = 0$, we get 
	\begingroup\allowdisplaybreaks
 \begin{align}
2 \varphi^t(\X) &= \inp{\nabla \varphi^t(\X)}{ \X - \hat{\X} }\notag\\
 &\stackrel{\eqref{restricted strong convexity}}{\geq}  \frac{\mu_t L_t}{\mu_t + L_t}\normfro{\X - \hat\X}^2 + \frac{1}{\mu_t + L_t} \normfro{\nabla \varphi^t(\X)}^2 \notag\\
 &\geq \mu_t \normfro{\X - \hat\X}^2,\notag
 \end{align}\endgroup
 where the last inequality   follows from    $\normfro{\nabla \varphi^t(\X)}\geq \mu_t\normfro{\X - \hat\X}$.
The conclusions are obtained by using \cref{lem:distance between Euclideanmean and IAM}.

\end{proof}

\begin{proof}[\textbf{Proof of   \cref{lem:regul_cond}}]
	(1). 	Combining \eqref{ineq:RSI first look} with \eqref{ineq:bound type 1}, we get 
	\begin{align}\label{reg:LHS}
	%	&\sum_{i=1}^N  \inp{  \bar{x}-x_{i} }{\p_{T_{x_{i}}}\sum_{j=1}^N W_{ij}x_{j}} \notag \\
	\inp{\X - \bar \X}{\grad \varphi^t(\X)}  
	\geq  \varphi^t(\X) \cdot(2-  \normfroinf{\X- \bar\X}^2).
	 \end{align}
	 Since $\X\in\N_{R,t}$, invoking \eqref{ineq:quadratic growth_local} in \cref{prop:quadratic growth}, we get
	\[  \inp{\X- \bar\X }{\grad  \varphi^t(\X)}\geq  (1- 4r\delta_{1,t}^2) (1- \frac{\delta_{2,t}^2}{2}) \mu_t \normfro{\X-\bar\X}^2,  \]	
	where using the conditions \eqref{delta_1_and_delta_2}  completes the proof. \\
%	\note{Can we get tighter bound for  $\textcircled{3}$? } \\
 (2). For $\X\in\N_{l,t},$ combining \eqref{ineq:RSI first look}, \eqref{ineq:bound type 2} and \eqref{ineq:quadratic growth_local} yields
	\begin{align*}  &\quad \inp{\X- \bar\X }{\grad  \varphi^t(\X)}\\
	&\geq    [\mu_t (1- 4r\delta_{3,t}^2) -  \varphi^t(\X)  ]\normfro{\X-\bar\X}^2\\
	& \geq   { \half \mu_t} \normfro{\X-\bar\X}^2, \end{align*}
	where we used the conditions in \eqref{delta_3}. \\
\end{proof}

\begin{proof}[\textbf{Proof of \cref{prop:gradient dominant}}]
% 	Note that $\varphi^t(\bar \X) = 0$ and $\grad \varphi^t(\bar\X)=0$.
% 	By \eqref{ineq:lipschitz} and \eqref{ineq:error bound}, one has
% 	\[ \varphi^t(\X) \leq \frac{L_t}{2}\normfro{\X-\bar\X}^2\leq \frac{2L_t}{\mu_t^2}\normfro{\grad  \varphi^t(\X)}^2.   \] 
	By \eqref{ineq:RSI first look}, we get
\be\label{ineq:RSI second look}
\bad
\quad 2\varphi^t(\X) &=   \inp{\grad \varphi^t(\X)} {\X-  \bar\X}  +  \sum_{i=1}^N\inp{p_i }{q_i }\\
& \stackrel{\eqref{ineq:error bound}}{\leq} \frac{2}{\mu_t}\normfro{\grad \varphi^t(\X)}^2 +  \sum_{i=1}^N\inp{p_i }{q_i }.
\ead \ee
If $\X \in\N_{R,t},$ we use \eqref{ineq:bound type 1} to get
\[ (2-\delta_{2,t}^2)\varphi^t(\X)\leq \frac{2}{\mu_t}\normfro{\grad \varphi^t(\X)}^2.\]
If $\X \in \N_{l,t},$ we use \eqref{ineq:bound type 2} to get 
{\small\[ 2\varphi^t(\X)\leq \frac{2}{\mu_t}\normfro{\grad \varphi^t(\X)}^2 + \frac{\mu_t}{4}\normfro{\X - \bar\X}^2\stackrel{\eqref{ineq:error bound}}{\leq}\frac{3}{\mu_t}\normfro{\grad \varphi^t(\X)}^2.\]}
We conclude the proof by noting $\delta_{2,t}\leq 1/6.$
\end{proof}
\begin{proof}[\textbf{Proof of Lemma \ref{lem:bound_of_grad}}]
	First, using \eqref{property of projection onto tangent} we have 
		\be\label{rewrite}
		\grad \varphi^t_i(\X) = x_i - \sum_{j=1}^N W_{ij}x_j - \half x_i  \sum_{j=1}^N W_{ij}^t(x_{i}-x_{j})^\top (x_{i}-x_{j}) .
		\ee
	Since $\sum_{i=1}^N  \nabla \varphi^t_i(\X) = \sum_{i=1}^N(x_i - \sum_{j=1}^N W_{ij}x_j)= 0$,  we have
	\begingroup\allowdisplaybreaks
	{	
		\begin{align*}
		\quad  &\normfro{ \sum_{i=1}^{N}\grad  \varphi^t_i(\X)}=     \half    \normfro{  \sum_{i=1}^{N}  x_{i}\sum_{j=1}^N W_{ij}^t(x_{i}-x_{j})^\top (x_{i}-x_{j})}\\
		&\leq \half  \sum_{i=1}^{N} \normfro{  \sum_{j=1}^N W_{ij}^t(x_{i}-x_{j})^\top (x_{i}-x_{j}) }\\
		&\leq  \half \sum_{i=1}^{N} \sum_{j=1}^N W_{ij}^t\normfro{  x_{i}-x_{j}  }^2= 2 \varphi^t(\X)\leq L_t\normfro{\X-\bar\X}^2,
		%	&\leq \frac{1}{ N} \sum_{i=1}^{N} \sum_{j=1}^N 2W_{ij}^t( \normfro{  x_{i}- \bar x  }^2 + \normfro{  x_{j}- \bar x  }^2)\\
		%	&\leq \frac{4}{N}\normfro{\X -\bar \X}^2.
		\end{align*}  
	}\endgroup
	where the last inequality follows from  \eqref{ineq:lipschitz}. 
	%	Secondly,  using \eqref{ineq:lips_riemanniangrad} gives
	%	\[ \normfro{\grad h^t(x_i) - \grad h^t(\bar x)} \]
	%	 Secondly, since $\X\in\N$. Similar as the proof of \cref{lem:regul_cond}, we have
	%	\begin{align*}
	%	&\inp{\bar\X-\X}{\grad h^t(\X)}\\
	%	=	& \sum_{i=1}^N  \inp{ \bar{x}-x_{i}}{\sum_{j=1}^N W_{ij}^t  x_{j} - x_i  + \half x_i\sum_{j=1}^N W_{ij}^t(x_i-x_j)^\top (x_i-x_j)}  \\
	%	\leq & \sigma_2^t\normfro{\X-\bar\X}^2 + (1+2\sqrt{r}\delta_1)\normfro{\X-\bar\X}^2 + 2\delta_2\normfro{\X-\bar\X}^2\\
	%	\leq & (1 + \sigma_2^t + \frac{5}{2}\delta_2)\normfro{\X-\bar\X}^2\\
	%	\leq & 2\normfro{\X-\bar\X}^2,
	%	\end{align*}
	%	where the last two inequalities follow from $\delta_1\leq\frac{1}{4\sqrt{r}}\delta_2$ and $\delta_2\leq \frac{1-\sigma_2^t}{20}$.
	%	Therefore, we get
	%	\[ 	\normfro{\grad h^t(\X)} \leq 2(1-\min_{i\in[N]} W_{ii})\normfro{\X-\bar\X}.  \]
	%	Now, suppose $\X\in\N$, we get $\normfro{x_i - x_j} \leq \normfro{x_i - \bar x} + \normfro{x_j - \bar x}\leq 2\delta_2 \leq 1$ for any $i,j\in[N]$. So, $\half (x^\top_i x_j + x_j^\top x_i)$ is positive semi-definite by \cref{lem:help lemma}.  
	%	It follows from \eqref{ineq:lips_riemanniangrad-local} that 
	%	\[ 	\normfro{\grad h^t(\X)} \leq  \normfro{\X-\bar\X}.  \]
		Moreover, it is clear that in the embedded Euclidean space we have
	\begingroup\allowdisplaybreaks
	\begin{align}
	0&\leq \varphi^t(\X - \frac{1}{L_t}\nabla\varphi^t(\X))\notag\\
	&\stackrel{\eqref{ineq:lip_smooth_E}}{\leq} \varphi^t(\X) + \langle\nabla\varphi^t(\X),- \frac{1}{L_t}\nabla\varphi^t(\X)\rangle + \frac{1}{2L_t}\normfro{ \nabla\varphi^t(\X)}^2\notag\\
	& =  \varphi^t(\X) - \frac{1}{2L_t}\normfro{ \nabla\varphi^t(\X)}^2.\notag
	\end{align} \endgroup
	Since $\grad \varphi^t_i(\X) = \p_{\T_{x_i}\M}(\nabla \varphi^t_i(\X))$, we get
	\begin{align*}
	&\quad \normfro{\grad  \varphi^t(\X)}^2\leq  \normfro{\nabla \varphi^t(\X)}^2
	%	&=  4\sum_{i=1}\normfro{\sum_{j=1}^N W_{ij}^t(x_i-x_j)}^2\\
	%	&\leq 4\sum_{i=1}\sum_{j=1}^N W_{ij}^t\normfro{x_i-x_j}^2
	\leq 2L_t\cdot\varphi^t(\X).%\stackrel{\leq}{under}  4L\normfro{\X-\bar\X}^2 .
	\end{align*}
	%	\begin{align*}
	%	&\quad \normfro{\grad  \varphi^t(\X)}^2\\ &\leq \sum_{i=1}^N \left[ 2 \sum_{j=1}^N W^t_{ij}\normfro{x_{j} - x_i}^2 +  \frac{1}{2} \sum_{j=1}^N W^t_{ij}\normfro{x_i-x_j }^4\right]\\
	%	&\leq \sum_{i=1}^N \left[ 2 \sum_{j=1}^N W^t_{ij}\normfro{x_{j} - x_i}^2 +  2\delta_2^2 \sum_{j=1}^N W^t_{ij}\normfro{x_i-x_j }^2\right]\\
	%	&\leq   8  (1+\delta_2^2)\normfro{\X-\bar\X}^2 .
	%	\end{align*}
	Finally, it follows from $\X\in\N_{2,t}$     that
	\begingroup\allowdisplaybreaks
	\begin{align*}
	\normfro{\grad  \varphi^t_i(\X)}\leq \normfro{  \sum_{j=1}^N W_{ij}^t (x_{j} - x_i)}
	\leq 2\delta_{2,t}  . 
	\end{align*}
	\endgroup
\end{proof}

\begin{proof}[\textbf{Proof of \cref{lem:helpful lem to prove X in N_Qt}}]
	First, we prove it for $\X\in\N_{R,t}$.	It follows from \eqref{ineq:RSI first look} and \eqref{ineq:bound type 1} that
	{
		\begin{align*} 
		\bad
		  \inp{\X-\bar\X}{\grad \varphi^t(\X)} 
		\geq  \Phi_R\cdot\varphi^t(\X) . 
		\ead\end{align*}	
	}
	%	Since  it holds for each $i\in[N]$ that
	%	\begin{align*} &\quad   \normfro{\sum_{j=1}^N   W_{ij}^t   (x_i-x_j)^\top (x_i-x_j)}\\
	%	&	\leq  \sum_{j=1}^N   W_{ij}^t\normfro{x_i-x_j}^2\leq \varphi^t(\X) ,  \end{align*}
	%	we get
	%	\begin{align*}
	%	&\quad - \frac{1}{2}\inp{  ( x_{i} -\bar x)^\top (x_i - \bar x)    }{  \sum_{j=1}^N   W_{ij}^t   (x_i-x_j)^\top (x_i-x_j) }\\
	%	&\geq - \half \sum_{i=1}^N \normfro{ ( x_{i} -\bar x)^\top (x_i - \bar x) }\cdot\varphi^t(\X)\\
	%	&\geq - \half\normfro{\X-\bar\X}^2 \cdot\varphi^t(\X).
	%	\end{align*}
	%	From the condition $\normfroinf{\X-\bar\X}\leq \frac{\sqrt{1-\sigma_2^t}}{6}\leq 1$, we derive that
	%	\[  \inp{\X-\bar\X}{\grad \varphi^t(\X)} \geq \varphi^t(\X) . \]
	Combining with \eqref{ineq:bound-of-gradh},  we get 
	$\inp{\X-\bar\X}{\grad \varphi^t(\X)} \geq \frac{\Phi_R}{2L_t}\normfro{\grad \varphi^t(\X)}^2 .$\\
	Secondly, for $\X\in \N_{l,t}$,  we have the similar arguments by combining \eqref{ineq:RSI first look} with \eqref{ineq:bound type 2}. 	
	Furthermore, if $\X\in\N_{R,t}$ or $\X\in\N_{l,t}$, we notice that \eqref{ineq: lower bound by gradient and distance} is the convex combination of \eqref{ineq: lower bound by gradient} and \eqref{ineq:reg_condition}.
\end{proof}

\begin{proof}[\textbf{Proof of \cref{lem:total deviation from 1/N}}]
	Note that $W^t$ is doubly stochastic with $\sigma_2^t$ as the second largest singular value. As $\X\in\N_2$, it follows that $\normfro{x_i-\bar x}\leq \delta_{2,t}$ for all $i\in[N]$. We then have
		\begingroup\allowdisplaybreaks
	\begin{align}
	&\quad \max_{i\in[N]} \normfro{\sum_{j=1}^N (W^t_{ij}-1/N)x_j}\notag\\
	&=  \max_{i\in[N]} \normfro{\sum_{j=1}^N (W^t_{ij}-1/N)(x_j-\bar{x}) }\notag  \\
	&\leq \max_{i\in[N]} \sum_{j=1}^N| W_{ij}^t - 1/N|\delta_{2,t}\leq \sqrt{N} \sigma_2^t\delta_{2,t},\notag 
	\end{align}\endgroup
	where the last inequality follows from  the bound on the total variation distance between any row of $W^t$ and $\frac{1}{N}\textbf{1}_N^\top$ \cite[Prop.3]{diaconis1991geometric}\cite[Sec 1.1.2]{boyd2004fastest}. The conclusion is obtained by setting $t\geq \lceil \log_{\sigma_2}(\frac{1}{2\sqrt{N}})\rceil$. 
\end{proof}

 \begin{proof}[\textbf{Proof of \cref{lem:bound_of_two_average}}]
	Let $\hat{x}=\frac{1}{N}\sum_{i=1}^Nx_i$ and  $\hat{y}=\frac{1}{N}\sum_{i=1}^Ny_i$ be the  Euclidean average points of $\X$ and $\Y$.
	Then, $\bar x$ and $\bar y$ are the (generalized) polar factor \cite{li2002perturbation} of $\hat x$ and $\hat y$, respectively. 
	% Denote 
%	\[\bar R_1: = \argmin_{R\in O(r)}\normfro{\hat xR - a}^2\ \text{and}\  \bar R_2 := \argmin_{R\in O(r)}\normfro{\hat yR - a}^2, \]
 % \[\bar R  := \argmin_{R\in O(r)}\normfro{\hat yR - \bar x}^2, \]
%	{\rd where  $\bar R $ is the solution  to the orthogonal Procrustes  problem\cite{horn2012matrix}. By the definition of IAM, we have
%	\[  \bar y = \bar x \bar R ^T. \] }  We also let $\sigma_i(\hat{x})$ and $\sigma_i(\hat{y})$ be the $i$-th largest singular value of $\hat{x}$ and $\hat{y}$. 
  We have 
	\[ 
	 \sigma_r(\hat{x}) \stackrel{\eqref{ineq:reg_condition-2-3}}{\geq}   1-2\frac{\normfro{\X-\bar\X}^2}{N} \stackrel{(i)}{\geq} 1 - 2\delta_{1,t}^2>0,\]  
	where $(i)$ follows from $\X\in\N_{1,t}$.  Similarly, we have $\sigma_r(\hat{y})\geq 1 - 2\delta_{1,t}^2$ since $\Y \in\N_{1,t}.$ \\
	%	Let $E: = \hat y - \hat x$. By the convexity of function $\sigma_1(\cdot)$, we have 
%	\[\sigma_1(E)\leq \max_{i\in[N]}\sigma_1(y_i - x_i)\leq \normfroinf{\X-\Y}< ( \sigma_r(\hat x)+ \sigma_r(\hat y))/2,\]
%	where the last inequality follows by the assumption in the lemma.  
	 Then, it follows from \cite[Theorem 2.4]{li2002perturbation} that
	\[\normfro{\bar{y} - \bar{x}} \leq\frac{2}{ \sigma_r(\hat x) + \sigma_r(\hat y) }\normfro{\hat y - \hat x}\leq \frac{1}{1-2\delta_{1,t}^2}\normfro{\hat{x}-\hat y}. \] The proof is completed.
	%	Furthermore, we have
	%	\[ \normfro{\hat{z}}^2 = \normfro{\frac{1}{N}\sum_{i=1}^N(x_i-y_i)}^2 \leq \frac{1}{N}\normfro{\X-\Y}^2. \]
%	Therefore, we obtain the desired result. 
\end{proof}

We use \cref{lem:bound_of_two_average}     for the following lemma. 

\begin{lemma}\label{lem:bound_of_k_k+1}
	If $\X_k\in\N_{R,t}, \X_{k+1}\in\N_{1,t}$ and  $x_{i,k+1} = \Retr_{x_{i,k}}(-\alpha \grad  \varphi^t_i(\X_k)  )$, where $\delta_{1,t}$ and $\delta_{2,t}$ are given by \eqref{delta_1_and_delta_2}.     It follows that 
	\begin{align*}
	 \normfro{\bar x_k - \bar{ x}_{k+1}}
	\leq \frac{L_t}{1-2\delta_{1,t}^2}  \frac{ \alpha + 2M\alpha^2 L_t  }{N}  \normfro{\X_k - \bar\X_k}^2.\end{align*}
\end{lemma}
\begin{proof}

% 	Using \eqref{ineq:ret_nonexpansive} in  \cref{lem:nonexpansive_bound_retraction} and \eqref{ineq:bound-of-gradh-i} in \cref{lem:bound_of_grad}  yields

%   \begin{align*}\normfroinf{\X_{k+1}-  {\X}_k} &\leq \max_{i\in[N]}\normfro{-\alpha\grad  \varphi^t(x_{i,k})} \\
% 	&\leq 2\alpha \delta_{2,t}  < 1-2\delta_{1,t}^2,
% 	\end{align*}
% 	where the last inequality is due to $\alpha\leq 2$, $\delta_{1,t}\leq 1/30$ and $\delta_{2,t}\leq 1/6$. 
 From \cref{lem:nonexpansive_bound_retraction} and \cref{lem:bound_of_grad}, we   have 
\begin{align*}
 & \quad \normfro{\hat x_k - \hat x_{k+1}}\\
 &\leq  \normfro{\hat x_k- \frac{\alpha}{N}\sum_{i=1}^{N}  \grad  \varphi^t_i(\X_k)    -\hat x_{k+1} }    + \normfro{ \frac{\alpha}{N}\sum_{i=1}^{N}  \grad \varphi^t_i(\X_k)  }  \\
&\stackrel{\eqref{ineq:ret_second-order}}{\leq } \frac{M}{N} \sum_{i=1}^{N} \normfro{\alpha  \grad  \varphi^t_i(\X_k)   }^2 + \alpha \normfro{ \frac{1}{N}\sum_{i=1}^{N}  \grad  \varphi^t_i(\X_k) } \\
%&\leq \frac{ 2M\alpha^2}{N}\normfro{\grad h^t(\X_k)}^2 + \alpha \normfro{ \frac{1}{N}\sum_{i=1}^N \grad h^t(x_{i,k})}\\ 
&\stackrel{\eqref{ineq:bound-of-sum-gradh}}{\leq }   \frac{ 2L_t^2M\alpha^2 + L_t\alpha }{N}  \normfro{\X_k - \bar\X_k}^2.
\end{align*}
Therefore, it follows from \cref{lem:bound_of_two_average} that
\begin{align*}
\normfro{\bar x_k - \bar{ x}_{k+1}} \leq \frac{1}{1-2\delta_{1,t}^2}\cdot \normfro{\hat x_k - \hat x_{k+1}} 
&\leq  \frac{L_t}{1-2\delta_{1,t}^2} \frac{\alpha + 2M\alpha^2 L_t  }{N}  \normfro{\X_k - \bar\X_k}^2 .
\end{align*}
\end{proof}

\begin{proof}[\textbf{Proof of Lemma \ref{lem:stayinbox}}]
	 First, we verify that $\X_{k+1}\in\N_{1,t}$.  Since $\X_k\in\N_{R,t}$, it follows from \cref{lem:regul_cond} that
   \begingroup\allowdisplaybreaks
   {
  	\be\label{stay in N2}
	\begin{aligned}
	&\quad\normfro{\X_{k+1} - \bar\X_{k+1}}^2\leq \normfro{\X_{k+1} - \bar\X_{k}}^2 \\
%	&\leq \normfro{\X_{k+1} - \bar\X_{k}}^2\\
	&\leq \sum_{i=1}^N \normfro{ x_{i,k} - \alpha \grad \varphi^t_i(\X_k)  - \bar x_k }^2 \\
	&= \normfro{\X_{k} - \bar\X_k}^2 - 2\alpha\inp{\grad  \varphi^t(\X_{k})}{  \X_k-\bar \X_k} 
	 +   \normfro{\alpha\grad  \varphi^t(\X_k)}^2\\
	&\stackrel{\eqref{ineq:  lower bound by gradient and distance}}{\leq }  \left(1- 2\alpha(1-\nu) \gamma_{R,t}\right) \normfro{\X_{k} - \bar\X_k}^2   + \left(\alpha^2 - \frac{\alpha\nu\Phi }{L_t}\right) \normfro{\grad  \varphi^t(\X_k)}^2,
	\end{aligned}
	\ee
}\endgroup
for any $\nu\in [0,1]$, where the last inequality holds by noting $\Phi \geq 1$ for $\X\in\N_{R,t}$. 
%	By invoking \eqref{ineq:bound-of-gradh-local} in \cref{lem:bound_of_grad}, we get
%Then we   mimic the proof of \cite[Theorem 5]{zhang2013gradient}.
By letting $\nu = 1$ and $\alpha \leq \frac{\Phi}{L_t}$, we get
	\be\label{ineq:stay_in N1} 	\normfro{\X_{k+1} - \bar\X_{k+1}}^2\leq   \normfro{\X_k-\bar\X_k}^2. \ee
% 	Since $\alpha\leq 1$, we get $\normfro{\X_{k+1} - \bar\X_{k+1}}^2\leq \normfro{\X_k-\bar\X_k}^2$ 
and thus $\X_{k+1}\in\N_{1,t}$. \\
	Next, let us verify $\X_{k+1}\in\N_{2,t}$.  
	For each $i\in[N]$, one has 
	{\small
	\begingroup\allowdisplaybreaks
	\begin{align}
	% &\normfro{x_{i,k+1} - \bar{x}_{k+1}} \\
	& \normfro{x_{i,k+1} - \bar{x}_{k }}  \notag \\
	\stackrel{\eqref{ineq:ret_nonexpansive}}{\leq} & \normfro{x_{i,k} - \alpha \grad  \varphi^t_i(\X_k)   -\bar{x}_{k } }\notag\\
	\stackrel{\eqref{rewrite}}{=} &  \normfro{ (1-\alpha)(x_{i,k} - \bar{x}_k) + \alpha (\hat{x}_k- \bar{x}_{k })+\alpha\sum_{j=1}^N{W}_{ij}^t (x_{j,k} - \hat{x}_k) 
	+ \frac{\alpha}{2}x_{i,k}\sum_{j=1}^N W_{ij}^t(x_{i,k}-x_{j,k})^\top(x_{i,k}-x_{j,k})   }\notag\\
	\leq &(1-\alpha)\delta_{2,t} + \alpha \normfro{   \hat x_{k} - \bar{x}_k }
	 +\alpha \normfro{ \sum_{j=1}^N ({W}_{ij}^t -\frac{1}{N}) x_{j,k}  } 
	+   \frac{1}{2}\normfro{ {\alpha} \sum_{j=1}^N W_{ij}^t(x_{i,k}-x_{j,k})^\top(x_{i,k}-x_{j,k}) }\notag\\ 
	%\leq & \delta_0 - \alpha  (1-\sigma_2) \delta_0 +   \beta_kD+  2\alpha  \delta_0^2.\notag
	\stackrel{\eqref{key}}{\leq }& (1-\alpha)\delta_{2,t} + 2\alpha   \delta_{1,t}^2\sqrt{r}    +\alpha \normfro{ \sum_{j=1}^N ({W}_{ij}^t -\frac{1}{N}) x_{j,k} }   +2{\alpha}\delta_{2,t}^2 \notag\\
	\stackrel{\eqref{lem:bound_of_2_infty}}{\leq } &(1-\frac{\alpha}{2})\delta_{2,t} + 2 \alpha   \delta_{1,t}^2 \sqrt{r}  + 2{\alpha}\delta_{2,t}^2.\notag
	\end{align}\endgroup
}

	%		 \textcolor{red}{ Conjecture:  With high probability, $ \normfro{ \sum_{j=1}^N ({W}_{ij} -\frac{1}{N}) x_{j,k}  }<\frac{\log N}{\sqrt{N}}$ for all $i\in[N]$ if $W$ and $x_j$ are independent. \\	 	
	%		 	  } 
	%$\normfroinf{  [(W - \frac{1}{N}\mathbf{1}_N\mathbf{1}_N^\top )\otimes I_d ] (\X_{k} - \bar{\X}_k) }<\delta_2 $}
	%	
	%	Using \cref{lem:bound_of_grad} gives \[
	%	\normfro{\X_k-\X_{k+1}}\leq \alpha\normfro{\grad h^t(\X_k)}\leq  \alpha \sqrt{8(1+\delta_2^2)}\delta_1\sqrt{N}\leq 3 \alpha \delta_1\sqrt{N}\] and  \[\normfroinf{\X_{k+1}-  {\X}_k} \leq \max_{i\in[N]}\normfro{\alpha\grad h^t(x_{k,i})}\leq 2\alpha(\delta_2 + \delta_2^2)\leq 1-\delta_1^2.\]
	Since $\alpha\geq 0$, by invoking \cref{lem:bound_of_k_k+1} we get
	{\small
	\begin{align*}
	 \normfro{\bar x_k - \bar{ x}_{k+1}} \leq L_t \cdot\frac{2 M\alpha^2L_t +  \alpha }{N(1-2\delta_{1,t}^2)}  \normfro{\X_k - \bar\X_k}^2  
\leq  \frac{  10\alpha \delta_{1,t}^2}{1-2\delta_{1,t}^2}, 
	\end{align*}
}
	where the last inequality follows  from $\alpha \leq   \frac{1}{ M} $ and $L_t\leq 2$. Therefore, using the conditions on $\delta_{1,t}$ and $\delta_{2,t}$ in \eqref{delta_1_and_delta_2} gives
		\begingroup\allowdisplaybreaks
	\begin{align} &  \normfro{x_{i,k+1} - \bar{x}_{k+1 }}
	\leq \normfro{x_{i,k+1} - \bar{x}_{k }} + \normfro{\bar{x}_{k} - \bar{x}_{k+1}} \notag \\
	\leq  &(1-\frac{\alpha}{2})\delta_{2,t} + 2\alpha   \delta_{1,t}^2 \sqrt{r}  + 2{\alpha}\delta_{2,t}^2 + \frac{ 10}{1-2\delta_{1,t}^2}\alpha\delta_{1,t}^2
	\leq \delta_{2,t}.\notag
	\end{align}\endgroup
	The proof is completed. 
\end{proof}

\begin{proof}[\textbf{Proof of \cref{thm:linear_rate_consensus}}]

 (1). Since $ 0< \alpha  \leq \min\{1, \frac{\Phi}{L_t}, \frac{1}{M}\}$. 	By \cref{lem:stayinbox}, we have $\X_k\in \N_{R,t}$ for all $k\geq 0$.  
	%Then we   mimic the proof of \cite[Theorem 5]{zhang2013gradient}.
By choosing any $\nu \in (0,1)$ and $\alpha \leq \frac{\nu\Phi}{ L_t}$, we get from \eqref{stay in N2} that
	\be\label{stay in N2 1} 	\normfro{\X_{k+1} - \bar\X_{k+1}}^2\leq  ( 1-    2\alpha(1-\nu)  \gamma_{R,t} )\normfro{\X_k-\bar\X_k}^2. \ee
	We know that $\X_k$ converges to the optimal set $\cX^*$ Q-linearly. 
	Furthermore, if $\alpha \leq \frac{2}{2MG+L_t}$, it follows from \cref{lem:sublinear-convergence} that the limit point of $\X_k$ is unique. Hence, $\bar\X_k$ also converges to a single point. \\
	(2).   If $\X_k\in\N_{l,t}$, we have the constant $\Phi = 2 - \half\normfro{\X-\bar\X}^2 > 1$ in \cref{lem:helpful lem to prove X in N_Qt}. So,  $\alpha \leq \frac{1}{L_t+2MG}\leq \frac{\Phi}{L_t}$, we have $\X_{k+1}\in\N_{l,t}$ by using the sufficient decrease inequality \eqref{ineq:suff-decrease}. The remaining proof follows the same argument of (1). 
\end{proof}

%%%%%%%%%%%%%%%%%%%%%%%%%%%%%%%%%%%%%%%%%%%%%%%%%%%%%%%%%%%%%%%%%%%%%%%%%%%%%%%%

\bibliographystyle{ieeetr}
\bibliography{manifold}

\end{document}